\documentclass{dcds}
\usepackage{amsmath,amsfonts,amsthm,amssymb,amscd,epsfig}
\usepackage{graphicx}
  \usepackage{paralist}
  \usepackage{epstopdf}
    \usepackage{yhmath}

 \usepackage[colorlinks=true]{hyperref}
\hypersetup{urlcolor=blue, citecolor=red}

\def\currentvolume{36}
 \def\currentissue{8}
  \def\currentyear{2016}
   \def\currentmonth{August}
    \def\ppages{X--XX}
     \def\DOI{10.3934/dcds.2016.8.xx}

\usepackage[mathscr]{eucal}
\usepackage{xspace}

\newtheorem{theorem}{Theorem}[section]

\theoremstyle{definition}

\newtheorem{remark}{Remark}

\newcommand{\R}{\mathbb{R}}
\newcommand{\D}{\mathbb{D}}

\newcommand{\MM}{\mathbb{M}}

\newcommand{\xx}{\mathbf{x}}
\newcommand{\zz}{\mathbf{z}}
\newcommand{\ff}{\mathbf{f}}
\newcommand{\kk}{\mathbf{k}}

\newcommand{\uu}{\mathbf{u}}

\newcommand{\ga}{\alpha}
\newcommand{\gb}{\mathbf{\beta}}
\newcommand{\gc}{\gamma}
\newcommand{\vph}{\varphi}
\newcommand{\gD}{\Delta}
\newcommand{\po}{\partial}

\newcommand{\ve}{\varepsilon}
\newcommand{\vae}{\varepsilon}

\newcommand{\gd}{\delta}

\newcommand{\gl}{\lambda}

\renewcommand{\O}{{\mathbf O}}

\renewcommand{\b}{\beta}

\newcommand{\Om}{\Omega}

\newcommand{\grad}{\nabla}

\newcommand{\bea}{\begin{eqnarray}}
\newcommand{\eea}{\end{eqnarray}}

\newcommand{\yy}{\mathbf{y}}

\newcommand{\nnu}{{\boldsymbol \nu}}

\numberwithin{equation}{section}
\numberwithin{figure}{section}

\title[Transonic Flows Past Curved Wedges]
{Transonic Flows with Shocks Past Curved Wedges for the Full Euler
Equations}

\author[Gui-Qiang Chen,  Jun Chen and  Mikhail Feldman]{}



 \email{chengq@maths.ox.ac.uk}
 \email{chenjun@sustc.edu.cn}
 \email{feldman@math.wisc.edu}

\thanks{$^*$ Corresponding author: Gui-Qiang G. Chen}

\dedicatory{Dedicated to Peter Lax on the occasion of his $90$th birthday\\ with admiration and affection}

\begin{document}
\maketitle
\centerline{\scshape Gui-Qiang Chen$^*$}
\medskip
{\footnotesize
 \centerline{Mathematical Institute, University of Oxford}
  \centerline{Andrew Wiles Building, Radcliffe Observatory Quarter, Woodstock Road}
   \centerline{ Oxford, OX2 6GG, UK}
}

\medskip

\centerline{\scshape Jun Chen }
\medskip
{\footnotesize
 \centerline{ Department of Mathematics, Southern University of Science and Technology}
   \centerline{Shenzhen, Guangdong 518055, P.R. China}
}

\medskip

\centerline{\scshape Mikhail Feldman }
\medskip
{\footnotesize
 \centerline{Department of Mathematics, University of Wisconsin-Madison}
   \centerline{ Madison, WI 53706--1388, USA}
}

\begin{abstract}
We establish the existence, stability, and asymptotic behavior of
transonic flows with a transonic shock past a curved wedge
for the steady full Euler equations in an important physical regime,
which form a nonlinear system of mixed-composite hyperbolic-elliptic type.
To achieve this, we first employ the transformation from Eulerian
to Lagrangian coordinates
and then exploit one of the new equations
to identify a potential function in Lagrangian coordinates.
By capturing the conservation properties of the system,
we derive a single second-order nonlinear elliptic equation for
the potential function in the subsonic region
so that the transonic shock problem is reformulated
as a one-phase free boundary problem for the
nonlinear equation with the shock-front as a free boundary.
One of the advantages of this approach is that,
given the shock location or equivalently the entropy function
along the shock-front downstream,
all the physical variables can be expressed as functions of the gradient
of the potential function, and the downstream asymptotic behavior
of the potential function at infinity can be uniquely determined
with a uniform decay rate.
To solve the free boundary problem,
we employ the hodograph transformation to transfer the free boundary
to a fixed boundary,
while keeping the ellipticity of the nonlinear equation, and
then update the entropy function to prove that the updating map
has a fixed point.
Another advantage in our analysis is in the context of the full Euler equations
so that
the Bernoulli constant is allowed to change for different fluid trajectories.
\end{abstract}
\maketitle

\section{Introduction}
We are concerned with the existence, stability, and asymptotic behavior of
steady transonic flows with transonic shocks past curved wedges
for the full Euler equations. The two-dimensional steady, full Euler
equations for polytropic gases have the form
({\it cf.} \cite{Chorin,CF,Majda}):
\begin{equation}\label{Euler1}
\left\{\begin{array}{ll}
 \nabla\cdot (\rho \uu)=0, \\
\nabla\cdot\left(\rho{\uu\otimes\uu}\right)
  +\nabla p=0,\\
\nabla\cdot\left(\rho\uu(E+ \frac{p}{\rho})\right)=0,
\end{array}\right.
\end{equation}
where $\grad=\grad_{\xx}$ is the gradient in $\xx=(x_1,x_2)\in\R^2$,
$\uu=(u_1, u_2)$ the velocity, $\rho$ the density, $p$ the pressure, and
$
E=\frac{1}{2}|\uu|^2+e
$
the total energy with internal energy $e$.

Choose density $\rho$ and entropy $S$ as the independent thermodynamical variables.
Then the constitutive relations can be written as
$(e,p,\theta)=(e(\rho,S), p(\rho,S),\theta(\rho,S))$ governed by
$$
\theta dS=de-\frac{p}{\rho^2}d\rho,
$$
where $\theta$ represents the temperature.
For a polytropic gas,
\begin{equation}\label{gas-1}
p=(\gamma-1)\rho e,
\qquad e=c_v\theta, \qquad \gamma=1+\frac{R}{c_v}>1,
\end{equation}
or equivalently,
\begin{equation}\label{gas-2}
p=p(\rho,S)=\kappa\rho^\gamma e^{S/c_v}, \qquad e=e(\rho,
S)=\frac{\kappa}{\gamma-1}\rho^{\gamma-1}e^{S/c_v},
\end{equation}
where $R>0$ may be taken to be the universal gas
constant divided by the effective molecular weight of the particular
gas, $c_v>0$ is the specific heat at constant volume,
$\gamma>1$ is the adiabatic exponent, and $\kappa>0$ is any constant
under scaling.

The sonic speed
of the flow for polytropic gas is
\begin{equation}\label{1.4aa}
c=\sqrt{\frac{\gamma p}{\rho}}.
\end{equation}
The flow is subsonic if $|\uu| < c$ and supersonic
if $|\uu|> c$.
For a transonic flow, both cases occur in the flow, and then system \eqref{Euler1}
is of mixed-composite hyperbolic-elliptic type, which consists
of two equations of mixed elliptic-hyperbolic type
and two equations of hyperbolic type.

System \eqref{Euler1} is a prototype of general nonlinear systems
of conservation laws:
\begin{equation}\label{Euler1a}
\nabla_\xx\cdot \mathbf{F}(U)=0,  \qquad \xx\in \R^n,
\end{equation}
where $U: \R^n\to \R^m$ is unknown, while $\mathbf{F}: \R^m\to \MM^{m\times n}$
is a given nonlinear mapping for the $m\times n$ matrix space $\MM^{m\times n}$.
For \eqref{Euler1}, we may choose $U=(\uu, p, \rho)$.
The systems with form \eqref{Euler1a}
often govern time-independent solutions for multidimensional quasilinear
hyperbolic systems of conservation laws; {\it cf.} Lax \cite{Lax1,Lax2}.

It is well known that, for a steady upstream uniform supersonic
flow past a straight-sided wedge whose vertex angle is less than
the critical angle, there exists a shock-front emanating from the
wedge vertex so that the downstream state is either subsonic or
supersonic, depending on the downstream asymptotic condition at infinity
(see Appendix B and Fig. \ref{Figure-1} for the shock polar). The
study of two-dimensional steady uniform supersonic flows past a
straight-sided wedge can date back to the 1940s ({\it cf}.
Courant-Friedrichs \cite{CF};
also see Prandtl \cite{Prandtl} and von Neumann \cite{Neumann}).

For the case of supersonic-supersonic
shock ({\it i.e.} both states of the shock are supersonic),
local solutions
around the curved wedge vertex were first constructed by Gu
\cite{Gu}, Li \cite{Li}, Schaeffer \cite{Schaeffer}, and the
references cited therein. Global potential solutions are constructed
in \cite{Chen1,Chen2,CF,Zh2}
when the wedge has certain convexity or the wedge is a small
perturbation of the straight-sided wedge with fast decay in the flow
direction.
Furthermore, the stability and uniqueness of entropy solutions in $BV$
containing the strong supersonic-supersonic shock
were established in Chen-Li \cite{ChenLi}.

For the case of supersonic-subsonic shock ({\it i.e.} transonic
shock-front), the stability of these fronts under a perturbation of
the upstream flow, or a perturbation of wedge boundary, has been studied
in Chen-Fang \cite{ChenFang} for the potential flow and in Fang \cite{Fang}
for the Euler flow with a uniform Bernoulli constant.
In particular, the stability of
transonic shocks in the steady Euler flows with a uniform Bernoulli constant
was first established in the weighted Sobolev norms in Fang \cite{Fang},
even though the downstream asymptotic decay rate of the shock slope
at infinity was not derived.

In this paper, one of our main objectives is to deal with
the asymptotic behavior of
steady transonic flows with a transonic shock past a curved wedge
for the full Euler equations, especially the uniform decay rate
of the transonic shock slope and the subsonic flows downstream at infinity.
For a fixed uniform supersonic state $U_0^-$, there is an arc on the
shock polar corresponding to the subsonic states; see Fig. \ref{Figure-1}.
When the wedge angle is less than
the critical angle $\theta_{\rm w}^{\rm c}$,
the tangential point $T$ corresponding to the critical angle
divides arc $\wideparen{HS}$ into the two open arcs
$\wideparen{TS}$ and $\wideparen{TH}$.
The nature of these two cases is very different.
In this paper, we focus mainly
on the stability of transonic shocks in the important physical regime
$\wideparen{TS}$ when the wedge angle is
between the sonic angle $\theta_{\rm w}^{\rm s}$ and
the critical angle $\theta_{\rm w}^{\rm c}>\theta_{\rm w}^{\rm s}$.

To achieve this, we first rewrite the problem in Lagrangian coordinates
so that the original streamlines in Eulerian
coordinates become straight lines and the curved wedge boundary in
Eulerian coordinates becomes a horizontal ray in
Lagrangian coordinates. Then we exploit one of the new equations to
identify a potential function $\phi$ in Lagrangian coordinates.
By capturing the conservation properties of the Euler system,
we derive a single second-order nonlinear elliptic equation
for the potential function $\phi$ in the
subsonic region as in
\cite{CCF},
so that
the original transonic shock problem is
reformulated as a one-phase free boundary problem
for a second-order nonlinear elliptic equation with the shock-front as
a free boundary.
One of the advantages of this approach is
that, given the location of the shock-front,
or equivalently the entropy function
$A$ (which is constant along the fluid trajectories) along the shock-front
downstream,
all the physical variables $U=(\uu, p, \rho)$ can be expressed as functions of
the gradient of $\phi$, and the asymptotic behavior $\phi^\infty$ of
the potential $\phi$ at the infinite exit can be uniquely
determined.

To solve the free boundary problem, we have to determine the
free boundary, and both the subsonic phase and entropy function
defined in the downstream
domain with the free boundary as a part of its boundary.
We approach
this problem by employing the hodograph transformation to transfer the
free boundary to a fixed boundary, while keeping the ellipticity of the
second order partial differential equations, and then by updating the
entropy function $A$ to prove that the updating map
for $A$ has a fixed point.

For given entropy function $A$, we first determine {\it a priori}
the limit function of the
potential function downstream at infinity. Then we solve the
second-order elliptic equations for the potential function in the
unbounded domain with the fixed boundary conditions
and the downstream asymptotic condition at infinity.
This is achieved through the fixed
point argument by designing an appropriate map.
In order to define this map,
we first linearize the second-order elliptic equation for the
identified potential function $\phi$ based on the limit function
$\phi^\infty$ of $\phi$, solve the linearized problem
in the fixed region, and then make delicate estimates of the
solutions, especially the corner singularity near the intersection
between the fixed shock-front and the wedge boundary. These
estimates allow us to prove that the map has a fixed point
that is the subsonic solution in the
downstream domain. Finally, we prove that the entropy function $A$ is a
fixed point via the implicit function theorem.

Since the transformation between the Eulerian and Lagrangian
coordinates is invertible, we obtain the existence and uniqueness of
solutions of the wedge problem in Eulerian coordinates by
transforming back the solutions in Lagrangian coordinates,
which are
the real subsonic phase for the free boundary problem. The
asymptotic behavior of solutions at the infinite exit is also
clarified. The stability of transonic shocks and corresponding
transonic flows is also established by both employing the
transformation from Eulerian to Lagrangian coordinates
and developing careful, detailed estimates of the solutions.

Another advantage in our analysis here is in the context of the real
full Euler equations so that the solutions do not necessarily obey
Bernoulli's law with a {\it uniform} Bernoulli constant, {\it i.e.},
the Bernoulli constant is allowed to change for different fluid
trajectories (in comparison with the setup in
\cite{schen1,schen2,schen-yuan,Fang}).

By the closeness assumption of solution $U$ to the uniform flow
in the subsonic region, we obtain the asymptotic behavior of $U$ as
$y_1 \to \infty$. The asymptotic state
$U^{\infty} = (\uu^\infty, p^\infty, \rho^\infty)$ is uniquely
determined by state $U^-$ of
the incoming flow and the wedge angle at infinity.

We remark that,
when $U_0^+$ is on arc $\wideparen{TH}$ (see Fig. \ref{Figure-1} below),
the nature of the oblique boundary condition near the origin
is significantly different from the
case when $U_0^+$ is on arc $\wideparen{TS}$. Such a difference
affects the regularity of solutions at the origin in general.
It requires a further understanding of global features of the problem,
especially the global relation between the regularity near
the origin and the decay of solutions at infinity,
to ensure the existence of a $C^{1,\alpha}$ solution.
A different approach may be required to handle this case,
which is currently under investigation.

The organization of this paper is as follows: In \S 2, we first
formulate the wedge problem into a free boundary problem
and state the main theorem.

In \S 3, we reduce the Euler system into a second-order
nonlinear elliptic equation in the subsonic region
and then reformulate the wedge problem
into a one-phase free boundary problem
for the second-order nonlinear elliptic equation
with the shock-front as a free boundary.

In \S 4,  we use the hodograph transformation to make the free boundary
into a fixed boundary, in order to reduce the difficulty of the free
boundary. After that, we only need to solve for the unknown entropy
function $A$ as a fixed point.

In \S 5, for a given entropy function $A$,
we solve the reformed fixed boundary
value problem in the unbounded domain and determine {\it a priori}
the downstream asymptotic function of the potential function at infinity.
Then, in \S 6, we prove that the entropy function $A$ is a
fixed point via the implicit function theorem,
which is one of the novel ingredients in this paper.

In \S 7, we determine the decay of the solution to the asymptotic
state in the physical coordinates.

In \S 8, we establish the stability of the transonic solutions
and transonic shocks under small perturbations of the incoming flows
and wedge boundaries. We finally give some remarks for the problem
when the downstream state of the background solution is on arc
$\wideparen{TH}$ in \S 9. In Appendices, we show two comparison
principles and derive a criterion for different arcs
$\wideparen{TS}$ and $\wideparen{TH}$ on the shock polar, which are
employed in the earlier sections.

Finally, we remark that the stability of conical shock-fronts in three-dimensional flow
has also been studied in the recent years.
The stability of conical supersonic-supersonic shock-fronts
has been studied in Liu-Lien \cite{LL} in the class of BV solutions
when the cone vertex angle is small, and Chen \cite{Chen2003}
and Chen-Xin-Yin \cite{CXY} in the
class of smooth solutions away from the conical shock-front when the perturbed
cone is sufficiently close to the straight-sided cone.
The stability of three-dimensional conical transonic shock-fronts
in potential flow has been established in Chen-Fang \cite{CFang}
with respect to the conical perturbation
of the cone boundary and the upstream flow in appropriate function spaces.
Also see Chen-Feldman \cite{CFeldman1}.

\section{Mathematical Setup and the Main Theorem}

In this section, we first formulate the wedge problem into
a free boundary problem for the composite-mixed Euler equations,
and state the main theorem.

As is well-known, for a uniform horizontal incoming flow
$U_0^-=(u_{10}^-,0, p_0^-, \rho_0^-)$ past a straight wedge with
half-wedge angle $\theta_0$, the downstream constant flow can be
determined by the Rankine-Hugoniot conditions, that is, the shock
polar (see Appendix B and Fig.  \ref{Figure-1}).  According to the shock polar, the
two flow angles are important: One is the critical angle $\theta_{\rm w}^{\rm c}$
that ensures the existence of the attached shocks at the wedge
vertex, and the other is the sonic angle $\theta_{\rm w}^{\rm s}$ for which the
downstream fluid velocity at the sonic speed in the direction.
When the straight wedge angle $\theta_{\rm w}$ is between
$\theta_{\rm w}^{\rm s}$ and $\theta_{\rm w}^{\rm c}$,
there are two subsonic solutions; while the wedge angle
$\theta_{\rm w}$ is smaller than $\theta_{\rm w}^{\rm s}$,
there are one subsonic
solution and one supersonic solution.
We focus on the subsonic
constant state $U_0^+=(\uu_{0}^+, p_0^+, \rho_0^+)$
with $\uu_0^+\cdot (\sin\theta_0, -\cos\theta_0)=0$.
Then the transonic shock-front $\mathcal{S}_0$ is also straight,
described by $x_1= s_0 x_2$.
The question is whether the transonic shock
solution is stable under a perturbation of the incoming supersonic
flow and the wedge boundary.

Assume that the perturbed incoming flow $U^-$ is close to $U_0^-$,
which is supersonic and almost horizontal, and the wedge is close
to a straight wedge. Then, for any suitable wedge angle (smaller
than a critical angle), it is expected that there should be a shock-front
which is attached to the wedge vertex. If the subsonicity
condition is imposed in the far field downstream after the shock-front,
then the flow $U$ between the shock-front and the wedge should be subsonic.
Since the upper and lower subsonic regions do not interact with each other,
it suffices to study the upper part.

We now use a function $b(x_1)$ to describe the wedge boundary:
\begin{equation}\label{wall}
\partial\mathcal{W}=\{\xx\in \R^2\,:\,  x_2= b(x_1),\  b(0)=0\}.
\end{equation}

Along the wedge boundary $\partial\mathcal{W}$, the slip condition is
naturally prescribed:
\begin{equation}\label{slipcon}
\left. \frac{u_2}{u_1}\right|_{\partial\mathcal{W}} = b'(x_1).
\end{equation}
Let the shock-front $\mathcal{S}$ be $x_1 = \sigma(x_2)$ with
$\sigma(0)=0$. Then the domain for the subsonic flow is denoted by
\begin{equation}\label{domain:1}
\Omega_{\mathcal{S}}=\{\xx\in \R^2\,:\,  x_1> \sigma(x_2),\,  x_2 > b(x_1)\},
\end{equation}
and the shock-front $\mathcal{S}$ becomes a free boundary connecting the
subsonic flow (elliptic) with the supersonic flow (hyperbolic).

To be a weak solution of the Euler equations \eqref{Euler1},  the
Rankine-Hugoniot conditions should be satisfied along the shock-front:
\begin{equation} \label{con-RH}
\begin{cases}
[\,\rho u_1\,]=\sigma'(x_2) [\,\rho u_2\,],\\[1mm]
[\,\rho u_1^2 + p\,]= \sigma'(x_2) [\,\rho u_1 u_2\,],\\[1mm]
[\,\rho u_1 u_2\,] =  \sigma'(x_2) [\,\rho {u_2}^2 + p\,],\\[1mm]
[\,\rho u_1(E+\frac{p}{\rho})\,] = \sigma'(x_2) [\,\rho u_2
(E+\frac{p}{\rho})\,],
\end{cases}
\end{equation}
as the free boundary conditions on $\mathcal{S}$, where $[\, \cdot\, ]$ denotes
the jump between the quantity of two states across the shock-front.

\smallskip
For a fixed uniform supersonic state $U_0^-$, there is an arc on the
shock polar corresponding to the subsonic states.
When the wedge angle is less than
the critical angle $\theta_{\rm w}^{\rm c}>\theta_{\rm w}^{\rm s}$,
the tangential point $T$ corresponding to the critical angle
divides arc $\wideparen{HS}$ into the two open arcs
$\wideparen{TS}$ and $\wideparen{TH}$.
The nature of these two cases is very different.

\begin{figure}
 \centering
\includegraphics[height=65mm]{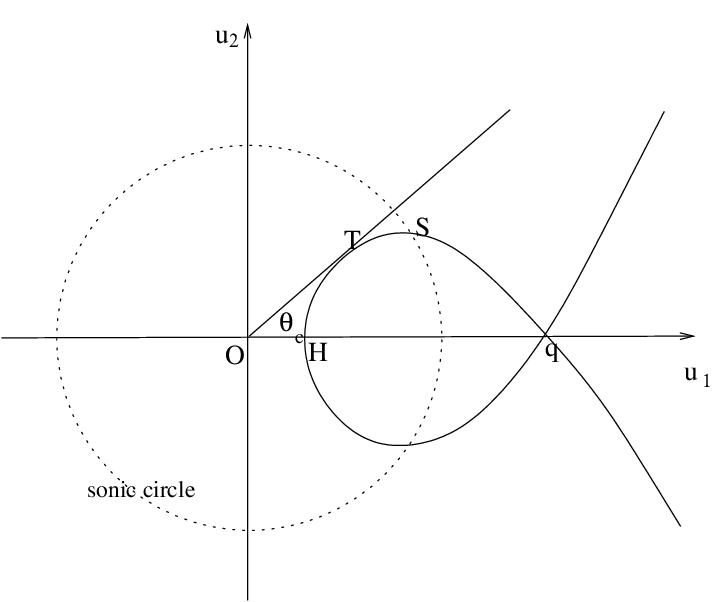}
\caption[]{Two arcs $\wideparen{TS}$ and $\wideparen{TH}$ on the
shock polar}\label{Figure-1}
\end{figure}

\smallskip
In this paper, we analyze the existence, stability, and asymptotic
behavior of steady transonic flows with a transonic shock in the
important regime  $\wideparen{TS}$ for the wedge angle $\theta_{\rm w}$.
To state our results,  we need the following weighed H\"older norms:
For any $\xx, \xx' $ in a two-dimensional domain $E$ and for a
subset $P$ of $\po E$, define
\begin{eqnarray*}
&\gd_\xx := \min\left\{ \mbox{dist}(\xx,P)  , 1\right\},
\qquad &\gd_{\xx,\xx'} :=\min\{\gd_\xx, \gd_{\xx'}\},\\
&\gD_\xx :=|\xx|+1,
\qquad\,\,\,\qquad &\gD_{\xx, \xx'}:= \min\{|\xx|+1, |\xx'|+1\}\\
&\widetilde{\gD}_{\xx} := \textrm{dist}(\xx,P)+1,
\qquad\,\,\,\qquad &\widetilde{\gD}_{\xx,\xx'}:= \min\{\widetilde{\gD}_{\xx}, \widetilde{\gD}_{\xx'}\}.
\end{eqnarray*}
Let $\ga \in (0,1)$, $\sigma, \tau, l \in \mathbb{R}$, and let $k$ be
a nonnegative integer. Let $\kk = (k_1, k_2)$ be an integer-valued
vector, where $k_1, k_2 \ge 0$, $|\kk|=k_1 +k_2$, and $D^{\kk}=
\po_{x_1}^{k_1}\po_{x_2}^{k_2}$. We define
\begin{eqnarray}
\nonumber
 &&[ f ]_{k,0;(\tau, l);E}^{(\sigma);P}
  := \sup_{\begin{subarray}{c}
\xx\in E\\
 |\kk|=k
\end{subarray}}\big(\gd_{\xx}^{\max\{k+\sigma,0\}} \gD_\xx^{\tau}\widetilde{\gD}_{\xx}^{l+k}  |D^\kk f(\xx)|\big), \\
 && {[ f ]}_{k,\ga;(\tau,l);E}^{(\sigma);P}
 := \sup_{
\begin{subarray}{c}
   \xx, \xx'\in E\\
 \xx \ne \xx', |\kk|=k
 \end{subarray}}
\Big(\gd_{\xx,\xx'}^{\max\{k+\ga+\sigma,0\}}
\gD_{\xx,\xx'}^{\tau}\widetilde{\gD}_{\xx,\xx'}^{l + k +\ga}\frac{|D^\kk f(\xx)-D^\kk f(\xx')|}{|\xx-\xx'|^\ga}\Big),
 \nonumber\\
&& \|f\|_{k,\ga;(\tau,l); E}^{(\sigma);P}
:= \sum_{i=0}^k {[f]}_{i,0 ;(\tau,l);E}^{(\sigma);P}
+ {[f]}_{k,\ga;(\tau,l);E}^{(\sigma);P}.
\label{def-norm}
\end{eqnarray}

For a vector-valued function $\ff=(f_1, f_2, \cdots,f_n )$, we
define
\[
\|\ff\|_{k,\ga;(\tau,l);E}^{(\sigma);P}
:=\sum_{i=1}^n  \|f_i\|_{k,\ga;(\tau,l);E}^{(\sigma);P}.
\]

For a function of one variable defined on $(0,\infty)$, we define
the H\"older norms with a weight at infinity. The definition
above can be reduced to one-dimensional if we keep only the weights
at  infinity. Then the notation becomes
$\|f\|_{k,\ga;(\tau);(0,\infty)}$.

We also need the norms with weights at infinity which apply only
for the derivatives:
\begin{equation}\label{normprime}
\|f\|_{k,\ga;(\tau,l);E}^{*, (\sigma);P}
:= \|f\|_{C^0(E)}+
   \|D f\|_{k-1,\ga;(\tau,l+1);E}^{(\sigma +1);P}.
\end{equation}
Similarly, the H\"{o}lder norms for a function of one variable
on $(0,\infty)$ with only the weights at infinity
are denoted by $\|f\|_{k,\ga;(\tau);(0,\infty)}^{*}$.

In terms of supersonic flows, we prescribe the initial data:
$$
U|_{\mathcal{I}}=U_0(x_2) \qquad\mbox{on}\,\,\,\, \mathcal{I}:=\{x_1 =0\}.
$$
Let $\Om^- $ be the domain for the incoming flows defined by
\begin{equation}\label{}
\Om^- = \left\{\xx\,:\, 0 <x_1 < 2 s_0 x_2 \right\}.
\end{equation}
For a given shock $\mathcal{S}=\{x_1=\sigma(x_2)\}$, let
\begin{equation}\label{}
    \Om_{\mathcal{S}}^- = \left\{\xx\,:\, 0 <x_1 < \sigma(x_2)
\right\}.
\end{equation}
We now fix parameters $\ga, \gb \in (0,1)$ with suitably small $\gb$, depending on the
background states.

\smallskip
Then we can conclude that
there is $\varepsilon>0$, depending on the background states,
such that, when
\begin{equation} \label{Uinitial}
 \|U_0-U^-_0\|_{2,\ga;(1+\gb) ;\mathcal{I}}< \varepsilon \qquad\mbox{for some $\beta>0$},
\end{equation}
there exists a constant $C_0>0$, independent of
$\vae$,  and a unique supersonic solution
$U^-=(\uu^-,p^-,\rho^-)(x,y)$ of system \eqref{Euler1} with the
initial condition $U^-|_{\mathcal{I}} = U_0$, well defined on $\Omega^-$,
such that
\begin{eqnarray}
\|U^--U^-_0\|_{2,\alpha;(1+\gb);\Omega^-}\le C_0 \|U_0 -
U_0^-\|_{2,\ga;(1+\gb);\mathcal{I}}. \label{supersonic:1}
\end{eqnarray}

This can be achieved by rewriting the problem
as an initial-boundary value problem
in the polar coordinates $(r, \theta)$
so that
system \eqref{Euler1}
is still a hyperbolic system,
domain $\Om^-$ becomes a half strip with $\theta$ time-like
and $r$ space-like, the initial data is
on $\{r>0, \theta=0\}$,
and the boundary data $v=0$ is on the
characteristic line
$\{r=0, 0\le \theta\le \arctan(2 s_0)\}$.
This is a standard initial-boundary value problem
whose almost-global existence of solutions can be obtained
as long
as $\vae$ is sufficiently small.

Assume that the wedge boundary satisfies
\begin{equation}\label{con-wallTS}
   \|b - b_0\|_{1,\ga;(\gb);\R_+}^{*} < \vae.
\end{equation}

\begin{theorem}[Main Theorem]\label{thm-mainTS}
Let the background solution $\{U_0^-, U_0^+\}$ satisfy that $U_0^+$ is
on arc $\wideparen{TS}$ in Fig. {\rm \ref{Figure-1}} for the straight
wedge boundary $x_2=b_0(x_1)=\tan\theta_0\, x_1, x_1>0$.
Then there is $\varepsilon>0$ such
that, when the initial data $U_0$ and the wedge boundary
$\partial \mathcal{W}=\{x_2=b(x_1), b(0)=0\}$
satisfy
\eqref{Uinitial} and \eqref{con-wallTS} respectively, there exist a
strong transonic shock $\mathcal{S}:=\{x_1= \sigma(x_2)\}$, a transonic
solution $\{U^-, U\}$ of the Euler equations \eqref{Euler1}
in $\Omega_{\mathcal{S}}$,
and an asymptotic downstream state
$U^\infty=(\uu^\infty, p_0^+, \rho^\infty)=V^\infty(x_2-\tan\theta_0\,x_1)$
for an appropriate function $V^\infty: [-\ve , \infty)\to \R^4$
with $\uu^\infty\cdot(\sin \theta_0, -\cos\theta_0)=0$
for the wedge angle $\theta_0$ such that

\smallskip
{\rm (i)} $U^-$ is a supersonic flow in $\Om^-_{\mathcal{S}}$, and $U$ is a
subsonic solution in $\Om_{\mathcal{S}}$;

\smallskip
{\rm (ii)} The Rankine-Hugoniot conditions \eqref{con-RH} hold along
the shock-front $\mathcal{S}$;

\smallskip
{\rm (iii)} The slip condition \eqref{slipcon} holds along the wedge
boundary $\partial\mathcal{W}$;

\smallskip
{\rm (iv)}
The following estimates hold:
\begin{eqnarray}
&& \|U^--U^-_0\|_{2,\alpha;(1+\gb,0);\Omega_{\mathcal{S}}^-}
 +\|U - U^+_0\|_{1, \ga;(0,\gb+1);\Om_{\mathcal{S}}}^{(-\ga); \partial\mathcal{W}}  +\|U - U^\infty\|_{0, \ga;(\gb,1);\Om_{\mathcal{S}}}
\nonumber \\
&&\quad +\|\sigma'(\cdot)- s_0\|^{*, (-1-\ga);\{0\}}_{2,\ga;(\gb);\R_+}
 +\|V^\infty-U_0^+\|_{0,\alpha; (1+\beta);[-\ve,\infty)}
\nonumber\\
&&\le C\left( \|U_0- U^-_0\|_{2,\ga;(1+\gb,0);\mathcal{I}} +
\|b-b_0\|_{1,\ga;(\gb);\R_+}^{*}\right), \label{est-U-per2}
\end{eqnarray}
where $C$ is a constant depending only on $U^0_{\pm}$, but
independent of $\ve$.

Moreover,  solution $U$ is unique within the class of
transonic solutions such that
the left-hand side of estimate \eqref{est-U-per2} is less
than $C\vae$.
\end{theorem}

\begin{remark}
Estimate \eqref{est-U-per2} implies that the
downstream flow and the transonic shock-front
are close to the background transonic solution.
The subsonic solution $U$ converges to $U^\infty$ at  rate $|\xx|^{-\beta}$ and the slope of the shock converges to the slope of the background shock at rate
 $|\xx|^{-\beta -1}$.
\end{remark}

\begin{remark}
Theorem {\rm \ref{thm-mainTS}}
indicates that
the asymptotic downstream state $U^\infty$ generally is not a uniform constant state.
If the $\xx$-coordinates are rotated with angle $\theta_0$ into the new coordinates $(\hat{x}_1, \hat{x}_2)$
so that the unperturbed wedge boundary $\partial\mathcal{W}_0=\{x_2-\tan\theta_0\, x_1\}$ becomes the $\hat{x}_1$-axis:
$$
(\hat{x}_1, \hat{x}_2)
=(\cos\theta_0\, x_1+\sin\theta_0\, x_2, -\sin\theta_0\, x_1+\cos\theta_0\, x_2),
$$
then $V^\infty=V^\infty(\hat{x}_2)$.
In Lagrangian coordinates, $\yy=(y_1, y_2)$, determined by \eqref{def-coord} in \S 3,
the asymptotic downstream state is a function of $y_2$ in general:
$U^\infty=U^\infty(y_2)=(\uu^\infty(y_2), p_0^+, \rho^\infty(y_2))$.
However, our argument also shows that, in the isentropic case with a constant Bernoulli quantity $B$ (see \eqref{eqn-bernoulli}),
the asymptotic state must be uniform and equal to the background state.
Also see Chen-Chen-Feldman \cite{CCF}.
\end{remark}

\section{Reduction of the Euler System and Reformulation of the Wedge Problem}

In this section, we first reduce the Euler system into a
second-order nonlinear elliptic equation and then reformulate the
wedge problem into a free boundary problem for the nonlinear
elliptic equation with the shock-front as the free boundary.

From the first equation in \eqref{Euler1}, there exists a unique
stream function $\psi$ in domain $\Om^- \cup \Om_{\mathcal{S}}$ such that
$$
\nabla \psi= (-\rho u_2, \rho u_1)
$$
with $\psi(\mathbf{0})=0$.

To simplify the analysis, we employ the following coordinate
transformation to the Lagrangian coordinates:
\begin{equation}\label{def-coord}
    \left\{
\begin{array}{lll}
y_1=x_1,\\[1mm]
y_2= \psi(x_1, x_2),
\end{array}
    \right.
\end{equation}
under which the original curved streamlines become straight. In
the new coordinates $\yy=(y_1, y_2)$, we still denote the unknown
variables $U(\xx(\yy))$ by $U(\yy)$ for simplicity of notation.

The original Euler equations in \eqref{Euler1} become the
following equations in divergence form:
\begin{eqnarray}\label{eqn-euler1}
  &&\big(\frac{1}{\rho u_1}\big)_{y_1}
    -\big(\frac{u_2}{u_1} \big)_{y_2}= 0,\\[1mm]
  &&\big( u_1 + \frac{p}{\rho u_1}\big)_{y_1}
   - \big( \frac{p u_2}{u_1}\big)_{y_2}= 0, \label{eqn-euler2}\\[1mm]
  &&(u_2)_{y_1} + p_{y_2}= 0,\label{eqn-euler3} \\[1mm]
  &&\big( \frac{1}{2}|\uu|^2 + \frac{\gc p}{(\gc-1)\rho} \big)_{y_1}=
  0.\label{eqn-euler4}
\end{eqnarray}

Let $\mathcal{T}: y_1= \hat{\sigma}(y_2)$ be a shock-front. Then,
from the above equations, we can derive the Rankine-Hugoniot
conditions along $\mathcal{T}$:
\begin{eqnarray}\label{con-RH1}
 &&\big[\frac{1}{\rho u_1}\big]=-\big[\frac{u_2}{u_1} \big] \hat{\sigma}'(y_2),\\[1mm]
 && \big[ u_1 + \frac{p}{\rho u_1}\big]=-\big[\frac{p u_2}{u_1}\big]
 \hat{\sigma}'(y_2),
  \label{con-RH2}\\[1mm]
&&[\,u_2 \,]= [\,p\,] \hat{\sigma}'(y_2),
\label{con-RH3}\\[1mm]
&& \big[ \frac{1}{2}|\uu|^2 + \frac{\gc p}{(\gc-1)\rho}\big]= 0.
\label{con-RH4}
\end{eqnarray}

The background shock-front now is $\mathcal{T}_0: y_1 = s_1 y_2$, with
$\frac{1}{s_1} = \rho_0^+ u_{10}^+(\frac{1}{s_0}- \tan \theta_0)>0$.
Without loss of generality, we assume that the supersonic solution
$U^-$ exists in domain $\D^-$ defined by
\begin{equation}\label{def-super-D}
    \D^- = \left\{\yy\,:\, 0 <y_1 < 2 s_1 y_2 \right\}.
\end{equation}

For a given shock function $\hat{\sigma}(y_2)$, let
\begin{eqnarray}
\D^-_{\hat{\sigma}}
&=&\left\{\yy: 0 <y_1 <\hat{\sigma}(y_2)\right\},
\label{3.11a}\\[1mm]
\D_{\hat{\sigma}}
&=& \left\{\yy: 0 <y_2 , \hat{\sigma}(y_2)< y_1
\right\}. \label{3.12a}
\end{eqnarray}

In either the supersonic or subsonic region,
$x_2$ can be solved as a function of $\yy$
since $\psi_{x_2} =\rho u_1 \ne 0$.
Let $x_2:= \phi(\yy)$ in the subsonic region $\D_{\hat{\sigma}}$ and
$x_2:=\phi^-(\yy)$ in the supersonic
region $\D^-_{\hat{\sigma}}$.
Given $U^-$, we can find the corresponding function $\phi^-$.
We now use function $\phi(\yy)$ to reduce the original Euler system to an elliptic
equation in the subsonic region.

By the definition of coordinate transformation \eqref{def-coord},
we have
\begin{equation}\label{eqn-phi-deri}
    \phi_{y_1} = \frac{u_2}{u_1} , \quad \phi_{y_2} = \frac{1}{\rho
    u_1},
\end{equation}
that is, $\phi(\yy)$ is the potential function of the vector field
$(\frac{u_2}{u_1}, \frac{1}{\rho u_1})$.

Equation \eqref{eqn-euler4} implies Bernoulli's law:
\begin{equation}\label{eqn-bernoulli}
\frac{1}{2}|\uu|^2 + \frac{\gc p}{(\gc-1)\rho} = B(y_2),
\end{equation}
where $B=B(y_2)$ is completely determined by the incoming flow $U^-$
at the initial position $\mathcal{I}$, because of the
Rankine-Hugoniot condition \eqref{con-RH4}.

From equations \eqref{eqn-euler1}--\eqref{eqn-euler4}, we find
\[
(\gc \ln \rho -\ln p)_{y_1} =0,
\]
which implies
\begin{equation}\label{eqn-rho-p}
    p= A(y_2) \rho^\gc \qquad\mbox{in the subsonic region $\D_{\hat{\sigma}}$}.
\end{equation}

\smallskip
With equations \eqref{eqn-phi-deri} and \eqref{eqn-rho-p}, we can
rewrite Bernoulli's law into the following form:
\begin{equation}\label{eqn-rho}
\frac{\phi_{y_1}^2 + 1 }{2 \phi_{y_2}^2} + \frac{\gc}{\gc
-1}A\rho^{\gc+1} = B \rho^2.
\end{equation}

In the subsonic region, $|\uu| < c:=\sqrt{\frac{\gc p}{\rho}}$. Therefore,
Bernoulli's law \eqref{eqn-bernoulli} implies
\begin{equation}\label{con-subsonic}
    \rho^{\gc-1} > \frac{2(\gamma-1)B}{\gamma(\gamma+1)A}.
\end{equation}

Condition \eqref{con-subsonic} guarantees that  $\rho$ can be
solved from \eqref{eqn-rho} as a smooth function of $(A,B,\grad
\phi)$.

\medskip
Assume that $A=A(y_2)$ has been known. Then $(\uu, p, \rho)$ can be
expressed as functions of $\grad \phi$:
\begin{equation}\label{eqn-express-U}
\rho = \rho(A, B, \grad \phi),\quad
\uu=(\frac{1}{\rho \phi_{y_2}}, \frac{\phi_{y_1}}{\rho \phi_{y_2}}),\quad  p=A\rho^\gc,
\end{equation}
since $B=B(y_2)$ is given by the incoming flow.

Similarly, in the supersonic region $\D^-$, we employ the corresponding
variables $( A^-, B, \phi^-)$ to replace $U^-$, where $B$ is the
same as in the subsonic region because of the Rankine-Hugoniot
condition \eqref{con-RH4}.

\smallskip
We now choose \eqref{eqn-euler3} to derive a second-order
nonlinear elliptic equation for $\phi$ so that the full Euler
system is reduced to this equation in the subsonic region.
Set
\begin{equation}\label{def-N}
N^1 =u_2,  \ \ \ N^2 =p .
\end{equation}
Then we obtain the {\it second-order nonlinear equation for $\phi$}:
\begin{equation}\label{eqn-N}
   (N^1)_{y_1} + (N^2)_{y_2}=0,
\end{equation}
where $N^i=N^i(A(y_2), B(y_2), \nabla\phi), i=1,2,$ are given by
\begin{eqnarray}\nonumber
N^1(A,B,\nabla\phi)&=&
\frac{\phi_{y_1}}{\phi_{y_2}\rho(A(y_2),B(y_2),\nabla\phi)},\\
N^2(A,B,\nabla\phi)&=&A(y_2)\rho(A(y_2),B(y_2),\nabla\phi)^\gamma
.\label{def-N12}
\end{eqnarray}

Let $q= \sqrt{u_1^2 + u_2^2}$. Then a careful calculation shows
that
\begin{eqnarray}
&&N^1_{\phi_{y_1}} =\frac{u_1 (c^2 - u_1^2)}{c^2 -q^2},\\
 \label{exp-N1phi2}&& N^1_{\phi_{y_2}} = N^2_{\phi_{y_1}} = -\frac{c^2\rho u_1u_2}{c^2 -q^2},\\
 && N^2_{\phi_{y_2}} = \frac{c^2\rho^2q^2u_1}{c^2 -q^2}.
\end{eqnarray}
Thus, the discriminant
\begin{equation}
N^1_{\phi_{y_1}} N^2_{\phi_{y_2}} -
N^1_{\phi_{y_2}}N^2_{\phi_{y_1}}= \frac{c^2\rho^2 u_1^2}{c^2 - q^2}
>0
\end{equation}
in the subsonic region with $\rho u_1\ne 0$.
Therefore, when solution $\phi$ is sufficiently close to $\phi_0^+$ (determined
by the subsonic background state $U_0^+$) in the $C^1$ norm,
equation \eqref{eqn-N} is uniformly elliptic, and the Euler system
\eqref{eqn-euler1}--\eqref{eqn-euler4} is reduced to the elliptic
equation \eqref{eqn-N} in domain $\D_{\hat{\sigma}}$,
where $\hat{\sigma}$ is the function for the transonic shock.

The boundary condition for $\phi$ on the wedge boundary $\{y_2 =
0\}$ is
\begin{equation}\label{Con-bd-phi-wall}
\phi (y_1, 0) = b(y_1).
\end{equation}

The condition on $\mathcal{T}$ is derived from the Rankine-Hugoniot
conditions \eqref{con-RH1}--\eqref{con-RH3}. Condition
\eqref{con-RH1} is equivalent to the continuity of $\phi$ across
$\mathcal{T}$:
\begin{equation}\label{continuity}
[\phi]|_{\mathcal{T}}=0.
\end{equation}
It also gives
\begin{equation}\label{for-tau-prime}
    \hat{\sigma}'(y_2)=-\frac{[\phi_{y_2}]}{[\phi_{y_1}]}(\hat{\sigma}(y_2), y_2).
\end{equation}

Replacing $\hat{\sigma}'(y_2)$ in \eqref{con-RH2} and
\eqref{con-RH3} with \eqref{for-tau-prime} gives rise to the
conditions on $\mathcal{T}$:
\begin{eqnarray}
&&G(U^-, A, \nabla \phi)\equiv
  [\phi_{y_1}]\Big[\frac{1}{\rho \phi_{y_2}}+A\rho^\gc \phi_{y_2}\Big]
   -[\phi_{y_2}][A\rho^\gc \phi_{y_1}] =0,  \label{eqn-G}\\[1mm]
&&H(U^-, A, \nabla\phi)\equiv
[\phi_{y_1}][N^1]+[\phi_{y_2}][N^2]=0.\label{eqn-H}
\end{eqnarray}

\smallskip
We will combine the above two conditions into the boundary
condition for \eqref{eqn-N} by eliminating $A$.

By calculation, we have
\begin{eqnarray}
&&N^1_A = \frac{\gc}{\gc -1}\frac{\rho^{\gc-1}u_2}{c^2 -q^2}, \\[1mm]
&& N^2_A = -\frac{\rho^\gc(q^2 + \frac{c^2}{\gc-1})}{c^2 -q^2}.
\end{eqnarray}
Thus, we obtain
\begin{eqnarray*}
  H_A &=& N^1_A [\phi_{y_1}] + N^2_A[\phi_{y_2}] \\
   &=& \frac{\gc}{\gc -1}\frac{\rho^{\gc-1}u_2}{c^2 -q^2}
   \Big[\frac{u_2}{u_1}\Big]
   - \frac{\rho^\gc(q^2 + \frac{c^2}{\gc-1})}
   {c^2 -q^2}\Big[\frac{1}{\rho u_1}\Big] \\
   &>& 0,
\end{eqnarray*}
and
\begin{eqnarray*}
G_A &=& [\phi_{y_1}]
  \Big(\frac{N^1_A}{\phi_{y_1}}
   + \phi_{y_2} N^2_A \Big)-[\phi_{y_2}]\phi_{y_1} N^2_A \\
  &=& \frac{u_2 \rho^\gc(q^2 + \frac{c^2}{\gc-1})}
   {u_1(c^2 -q^2)}\left[\frac{1}{\rho u_1}\right]
   - \frac{ \rho^{\gc -1}}{u_1(c^2-q^2)}
   \left({u_2}^2 + \frac{c^2 -u_1^2}{\gc -1}\right)\left[\frac{u_2}{u_1} \right]\\
   &<& 0,
 \end{eqnarray*}
since $[\frac{1}{\rho u_1}]<0$ and ${u_2}_-$ is close to $0$.

Therefore, both equations \eqref{eqn-G} and \eqref{eqn-H} can be solved
for $A$ to obtain
$A= g_1(U^-, \nabla \phi)$ and $A= g_2(U^-, \nabla\phi)$, respectively.
Then we obtain our desired condition
on the free boundary ({\it i.e.} the shock-front):
\begin{equation}\label{con-phi-shock}
\bar{g}(U^-,\nabla \phi):= (g_2 -g_1)(U^-,\nabla \phi) =0.
\end{equation}
Then the original transonic problem is reduced to the elliptic
equation \eqref{eqn-N} with the fixed boundary condition
\eqref{Con-bd-phi-wall} and
the free boundary conditions \eqref{continuity}
and \eqref{con-phi-shock},
and $A$ is determined through either of
\eqref{eqn-G}--\eqref{eqn-H}.

\section{Hodograph Transformation and Fixed Boundary Value Problem}

In order to reduce the difficulty of the free boundary, we employ the
hodograph transformation to make the shock-front into a fixed boundary.
After that, we only need to solve for the unknown function $A$.

\smallskip
We now extend the domain of $\phi^-$ from $\D^-$ to the first quadrant
$\D^- \cup \D_{\hat{\sigma}}$.
Let $\phi^-_0 = \frac{1}{\rho^-_0 u_{20}^-} y_2$, which is the
background potential function. We can extend $\phi^-$ into
$\D^- \cup \D_{\hat{\sigma}}$
such that
\begin{equation*}
\phi^- = \phi^-_0 \qquad  \mbox{when}\,\,\,
0 <2 s_1 y_2< y_1-1.
\end{equation*}

We then use the following partial hodograph transformation:
\begin{equation}\label{def-hodo}
 \left\{
\begin{array}{ll}
z_1=\phi - \phi^-,\\[1mm]
z_2= y_2,
\end{array}
    \right.
\end{equation}
so that $y_1$ is a function of
$(z_1, z_2)$: $y_1 = \varphi(z_1, z_2)$.

Let
\begin{eqnarray*}
  M^1 (U^-, A, \grad \phi)&=& N^1(A,B, \grad \phi)
  + N^2(A,B, \grad \phi)\frac{[\phi_{y_2}]}{[\phi_{y_1}]}, \\
  M^2 (U^-, A, \grad \phi)  &=&\frac{N^2(A,B, \grad
  \phi)}{[\phi_{y_1}]},
\end{eqnarray*}
and
  \begin{eqnarray*}\label{def-Mbar}
 && \overline{M}^i (\zz, A, \varphi, \grad \varphi)\\
 &&=-M^i\big(U^-(\varphi, z_2), A, \partial_{y_1}{\phi^-}(\varphi, z_2)
     +\frac{1}{\varphi_{z_1}}, \partial_{y_2}{\phi^-}(\varphi, z_2)
     -\frac{\varphi_{z_2}}{\varphi_{z_1}}\big), \quad i=1,2.
\end{eqnarray*}

Therefore, equation \eqref{eqn-N} becomes
\begin{equation}\label{eqn-Mbar}
   \big(\overline{M}^1 (\zz, A, \varphi, \grad \varphi)\big)_{z_1}
   +\big(\overline{M}^2 (\zz, A, \varphi, \grad \varphi)\big)_{z_2}=0.
\end{equation}
Notice that
\begin{eqnarray}
&&  \overline{M}^1_{\vph_{z_1}}  =  [\phi_{y_1}]^2
N^1_{\phi_{y_1}}
  + 2 N^1_{\phi_{y_2}} [\phi_{y_1}][
  \phi_{y_2}] + N^2_{\phi_{y_2}}[
  \phi_{y_2}]^2,\\[1mm]
&&\overline{M}^1_{\vph_{z_2}}  =  N^1_{\phi_{y_2}} [\phi_{y_1}]+
N^2_{\phi_{y_2}}[\phi_{y_2}]+ N^2, \label{exp-M12}\\[1mm]
&& \overline{M}^2_{\vph_{z_1}}  =  N^1_{\phi_{y_2}} [\phi_{y_1}]+
N^2_{\phi_{y_2}}[
  \phi_{y_2}]- N^2, \label{exp-M21}\\[1mm]
&& \overline{M}^2_{\vph_{z_2}}  =  N^2_{\phi_{y_2}}.
\end{eqnarray}
Also
\begin{eqnarray*}
\overline{M}^1_{\vph_{z_1}} \overline{M}^2_{\vph_{z_2}}  -
\frac{1}{4}\big(\overline{M}^1_{\vph_{z_2}}+ \overline{M}^2_{\vph_{z_1}}\big)^2=
[\phi_{y_1}]^2 \big(N^1_{\phi_{y_1}} N^2_{\phi_{y_2}} -
(N^1_{\phi_{y_2}})^2\big) >0,
\end{eqnarray*}
which implies that equation \eqref{eqn-Mbar} is uniformly elliptic, for any solution
$\varphi$ that is close to $\varphi_0^+$ (determined by the background solution $U_0^+$)
in the $C^1$ norm.

Then the unknown shock-front $\mathcal{T}$ becomes a fixed boundary, which
is the $z_2$-axis. Along the $z_2$-axis, condition
\eqref{con-phi-shock} is now
\begin{eqnarray}\label{con-phi-gt}
\tilde{g} (\zz, \varphi, \grad \varphi)\equiv
     \bar{g}\big(U^-(\varphi,z_2), \partial_{y_1}{\phi^-}(\varphi, z_2)
     +\frac{1}{\varphi_{z_1}}, \partial_{y_2}{\phi^-}(\varphi, z_2)
     -\frac{\varphi_{z_2}}{\varphi_{z_1}}\big)=0.
     \end{eqnarray}

We also convert condition \eqref{eqn-H} into the $\zz$-coordinates:
\begin{eqnarray}\label{con-Ht}
&&\widetilde{H}(\zz,A, \varphi, \grad \varphi)\nonumber\\
&&:= {H}(U^-(\varphi,z_2),A,
\partial_{y_1}{\phi^-}(\varphi, z_2)
     +\frac{1}{\varphi_{z_1}}, \partial_{y_2}{\phi^-}(\varphi, z_2)
     -\frac{\varphi_{z_2}}{\varphi_{z_1}})=0
\end{eqnarray}
along the $z_2$-axis.

\medskip
The condition on the $z_1$-axis can be derived from
\eqref{Con-bd-phi-wall} as follows: Restricted on $z_2 =0$, the
coordinate transformation \eqref{def-hodo} becomes
\begin{equation*}
    z_1 = b(y_1) - \phi_-(y_1, 0).
\end{equation*}
Then $y_1$ can be solved in terms of $z_1$ so that
\begin{equation}\label{con-phi-z1}
y_1 = \varphi(z_1, 0) = \widetilde{b}(z_1) .
\end{equation}

Let $Q$ be the first quadrant. Then the original wedge problem is
now reduced to both solving equation \eqref{eqn-Mbar} for $\varphi$ in
the unbounded domain $Q$ with the boundary conditions \eqref{con-phi-gt}
and \eqref{con-phi-z1} and determining $A$ via \eqref{con-Ht}.

\medskip
This will be achieved by the following fixed point arguments.
Consider a Banach space:
$$
X= \{\gl\,: \, \gl(0) =0, \|\gl\|_{1,\alpha; (1+\beta); (0, \infty)}^{(-\alpha); \{0\}}<\infty\}
$$
as defined in \eqref{def-space-x} below.
Then we define our iteration map $\mathcal{J}:  X\longrightarrow X$
through the following two steps:

\smallskip
{\bf 1.} Consider any $A=A(z_2)$ so that $A-w_t\in X$ satisfying
\begin{equation}\label{4.11-a}
\|A-A_0^+\|_{1,\alpha; (1+\beta); (0,\infty)}^{(-\alpha); \{0\}}\le C_0\varepsilon
\end{equation}
for some fixed constant $C_0>0$,
where $w_t=w_t(z_2)$ is determined by \eqref{6.2-a} below.
With this $A$, we solve equation \eqref{eqn-Mbar} for $\varphi=\varphi_A$ in
the unbounded domain $Q$ with the boundary conditions \eqref{con-phi-gt}
and \eqref{con-phi-z1} in a compact and convex set:
\begin{equation}\label{4.12-a}
\Sigma_\delta=\{\varphi :\, \|\varphi-\varphi^\infty\|_{2,\alpha;(\beta,0);Q}^{(-1-\alpha);\partial\mathcal{W}}\le \delta\}
\qquad\mbox{for sufficiently small $\delta>0$}
\end{equation}
in the Banach space:
\begin{equation}\label{4.13-a}
\mathcal{B} = \{\varphi  \,:\, \|\varphi-\varphi^\infty\|_{2,\ga';(\gb',0);Q}^{(-1-\ga');\partial\mathcal{W}} <
\infty\} \qquad \mbox{with $\,\,0< \ga'<\ga, \, 0<\beta'<\beta$},
\end{equation}
where $\varphi$ is determined by \eqref{eqn-varphi-inf}.  Equation \eqref{eqn-Mbar}
is uniformly elliptic for $\varphi  \in \Sigma_\gd$ for small $\delta>0$.
The existence of solution $\varphi_A \in \Sigma_\delta$ will be established
by the Schauder fixed point theorem in \S 5.

\smallskip
{\bf 2.} With this $\varphi=\varphi_A$, we solve \eqref{con-Ht} to obtain
a unique $\tilde{A}$ that defines $\mathcal{J}(A-w_t)=\tilde{A}-w_t$.

\medskip
Finally, by the implicit function theorem,
we prove that $\mathcal{J}$ has a fixed point $A-w_t$ in \S 6, for which $A$ satisfies
\eqref{4.11-a}.

\section{An Elliptic Problem to Determine $\varphi$ in Domain $Q$}
\label{sec-ell-pro}

In this section, for given $A$ satisfying \eqref{4.11-a},
we solve equation \eqref{eqn-Mbar}
for $\varphi$ in the unbounded domain $Q$ with boundary conditions
\eqref{con-phi-gt} and \eqref{con-phi-z1}. Before this, we determine {\it a priori}
the limit function $\varphi^\infty$ at infinity.

\subsection{Determine {\it a priori} the limit function $\varphi^\infty$  at infinity}
\label{Section-5.1}
First, we assume that the asymptotic downstream state $U^\infty$
depends only on $y_2$, which will be verified later.
Then we determine the limit function $\phi^\infty$ for $\phi$.
From \eqref{eqn-euler1}, we expect the flow direction at infinity
is the same as that of the
wedge. That is,
$$
\phi_{y_1} = \frac{u_2}{u_1} \to \tan \theta_0 =
\frac{u_{20}^+}{u_{10}^+} \qquad\,\,\,\mbox{as}\,\,\, y_1 \to \infty.
$$
Then
$$
\phi^\infty = \tan \theta_0\, y_1 + l(y_2).
$$
Replacing $\phi$ with $\phi^\infty$ in Bernoulli's law
\eqref{eqn-rho}, we obtain
\begin{equation*}
\frac{(\tan \theta_0)^2 + 1 }{2 {(l'(y_2))}^2}
+ \frac{\gc}{\gc-1}A\rho^{\gc+1} = B \rho^2.
\end{equation*}
From \eqref{eqn-euler3}, we expect that
pressure $p \to p^+_0$ and then
relation \eqref{eqn-rho-p} becomes
$ p^+_0 = A (\rho^\infty)^\gc $
so that $A=A(y_2)$
and $\rho^\infty(y_2)=(\frac{p_0^+}{A(y_2)})^{1/\gamma}$.
Therefore, the above equation
becomes
\begin{equation}\label{5.1a}
\frac{(\tan \theta_0) ^2 + 1 }{2 {(l'(y_2))}^2} + \frac{\gc}{\gc-1}
    A\,{\big(\frac{p^+_0}{A}\big)}^{(\gc+1)/\gc}
    = B{\big(\frac{p^+_0}{A}\big)}^{2/\gc}.
\end{equation}
This equation gives the expression for $l'(y_2)$. We can find
$l(y_2)$ by integration with $l(0) = w_0$, where $w_0$ is the limit
of $b- b_0$ as $y_1\to \infty$.

Then we employ
\begin{equation}\label{eqn-varphi-inf}
z_1=(\phi^\infty-\phi^-)(\varphi^\infty,z_2)
\end{equation}
to solve for $\varphi^\infty$.
Also, equation \eqref{eqn-varphi-inf} restricted
on $z_2 =0$ gives rise to
\[
z_1 =\tan \theta_0\,\tilde{b}_0 + w_0 - \phi^-(\tilde{b}_0 ,0),
\]
from which we can solve for $\tilde{b}_0$.

By the definition of $\varphi^\infty$, we know that $\varphi^\infty$
satisfies \eqref{eqn-Mbar}. That is,
\begin{equation}\label{eqn-minfty}
  \big(\overline{M}^1 (\zz, A, \varphi^\infty, \grad \varphi^\infty)\big)_{z_1}
   + \big(\overline{M}^2 (\zz, A, \varphi^\infty,
   \grad \varphi^\infty)\big)_{z_2}=0.
\end{equation}

\subsection{Linearization}
Let
\begin{equation} \label{def-sigma-delta}
 \Sigma_{\gd} = \big\{
w\,:\, \|w\|_{2,\ga;(\gb,0);Q}^{(-1-\ga);\partial\mathcal{W}} \le \gd \big\},
\end{equation}
where the wedge boundary $\partial\mathcal{W}$ is the $z_1$-axis.
We will omit $\partial\mathcal{W}$ in the norm
when no confusion arises.

To solve equation \eqref{eqn-Mbar} in the first quadrant $Q$, we
first linearize \eqref{eqn-Mbar} and solve the linearized equation in
bounded domains, and then take the limit to obtain a solution in the
unbounded domain $Q$.

For given $\vph$ such that $\vph - \vph^\infty \in \Sigma_{\gd}$, we
define a map $\mathcal{Q}$:
$\tilde{\varphi}-\varphi^\infty=:\mathcal{Q}(\vph - \vph^\infty)$
in $\Sigma_{\gd}$ as the solution of the linearized equation \eqref{eqn-v}
of the nonlinear equation \eqref{eqn-Mbar} at $\vph$ as described below,
and show that there exists a fixed
point that is a solution for equation \eqref{eqn-Mbar}.

We use a straight line $L^R:=\{z_2=R- z_1\}$ to cut off $Q$ into a
triangular domain $Q^R:=\{ 0<z_2 < R- z_1,z_1 >0 \} $.

Let
$$
v= \tilde{\varphi} - \varphi^\infty, \qquad
\zeta=\tilde{b}-\tilde{b}_0.
$$
Taking the difference of equations
 \eqref{eqn-Mbar} and \eqref{eqn-minfty} and linearizing the resulting
equation lead to
\begin{equation}\label{eqn-v}
   \sum_{i,j=1,2} (a^\varphi_{ij} v_{z_i} + b_j^\varphi v)_{z_j}=0,
\end{equation}
where
\begin{eqnarray}
&&a^\varphi_{ij}= \int_0^1 \overline{M} ^i_{\varphi_{z_j}}(\zz,A,
\varphi^\infty + s(\vph -\vph^\infty),
  \grad (\varphi^\infty+ s(\vph -\vph^\infty)))\,{\rm d}s,\label{def-aij}\\[1mm]
&&b_j^\varphi = \int_0^1 \overline{M}^j_{\varphi}(\zz,A,
\varphi^\infty + s(\vph -\vph^\infty), \grad (\varphi^\infty+ s(\vph
-\vph^\infty)))\, {\rm d}s \label{def-bi}
\end{eqnarray}
for  $i,j=1,2$.

Replace $\varphi$ by $\varphi^\infty$ in $a^\varphi_{ij}$ and denote by $a^\infty_{ij}$. Then $a^\varphi_{ij}- a^\infty_{ij}$ is bounded in the H\"{o}lder norm
$\|\cdot \|_{1,\ga;(\gb,1);Q}^{(-\ga);\partial\mathcal{W}}$, $b^\varphi_j$ is bounded in   $\|\cdot \|_{1,\ga;(\gb+2,0);Q}^{(-\ga);\partial\mathcal{W}}$.
Also, the uniform ellipticity of equation \eqref{eqn-v}
follows from \eqref{def-aij}
and the uniform ellipticity of \eqref{eqn-Mbar}
for the solutions close to $\varphi_0^+$, provided that
$\delta$ in \eqref{def-sigma-delta} is chosen
sufficiently small.

The boundary condition on the $z_1$-axis is
\begin{equation}\label{con-v-z1}
  v|_{z_2 = 0} = \zeta.
\end{equation}
On the cutoff line $L^R$, we prescribe the condition:
\begin{equation}\label{con-v-LR}
   v|_{L^R} = \zeta(R),
\end{equation}
which is compatible with the condition on the $z_1$-axis at point $(R,0)$.

\smallskip
Condition  \eqref{con-phi-gt} on the $z_2$-axis can be linearized
as follows: Condition \eqref{con-phi-gt} can be rewritten as
\begin{equation*}
    \tilde{g} (\zz, \varphi, \nabla \varphi )-
     \tilde{g} (\zz, \varphi^\infty, \nabla \varphi^\infty )
    = -\tilde{g} (\zz, \varphi^\infty, \nabla \varphi^\infty ).
\end{equation*}
Therefore, we derive the oblique condition:
\begin{equation}\label{con-obl-v}
    \sum_{i=1,2}\nu_i^\varphi v_{z_i} + c^\varphi v
    =-\tilde{g} (\zz, \varphi^\infty, \nabla \varphi^\infty ) =: g_0,
\end{equation}
where
\begin{eqnarray*}
&&\nu_i^\varphi = \int_0^1 \tilde{g}_{\varphi_{z_i}}(\zz,
\varphi^\infty + s(\vph -\vph^\infty),
  \grad (\varphi^\infty+ s(\vph -\vph^\infty)))\,{\rm d}s,\\[1mm]
&&c^\varphi = \int_0^1 \tilde{g}_{\varphi}(\zz,
\varphi^\infty + s(\vph -\vph^\infty),
  \grad (\varphi^\infty+ s(\vph -\vph^\infty)))\,{\rm d}s,
\end{eqnarray*}
which have all the corresponding bounded H\"{o}lder norms.

When $U_0^+$ is on arc $\wideparen{TS}$, the direction of
$\nnu=(\nu_1, \nu_2)$ is
\begin{eqnarray*}
\nu_1 &=& \tilde{g}_{\varphi_{z_1}}\\[1mm]
&=& \frac{-\rho^{\gc -1}}{(\gc -1)u_1^2 G_A H_A (c^2 -q^2)}
  \Big(\big[\frac{1}{\rho u_1}\big]^2 u_1^2 \rho^2 c^2 - 2c^2
   \rho u_1 u_2\big[\frac{1}{\rho u_1}\big]
    + u_2^2 (c^2 - u_1^2)\Big)\\[1mm]
&>& 0
\end{eqnarray*}
since $[\frac{1}{\rho u_1}]<0$,
and
\begin{eqnarray*}
  \nu_2= \tilde{g}_{\varphi_{z_2}}
 = \frac{-\rho^{\gc -1}u_2}{(\gc -1)u_1 G_A H_A (c^2 -q^2)}C_p,
\end{eqnarray*}
where
\begin{equation}\label{con-TS}
C_p=[p]\big(c^2 + (\gc -1)q^2 - \gc u_1^2\big) + (\gc -1)\rho q^2 u_2^2
 + \big[\frac{1}{\rho u_1} \big] \rho^2 c^2 u_1 q^2.
\end{equation}
Since, on arc $\wideparen{TS}$, $C_p<0$
from \eqref{B.11},
we have
$$
\nu_2 <0.
$$
In particular, if $\delta$ is small, then
$$
\nu_2 \le \frac{1}{2}\nu_{20}^+ < 0,
$$
where $\nu_{20}^+$ is the quantity $\nu_2$ for the background subsonic
state.
This implies that condition \eqref{con-obl-v} is uniformly oblique.

Set
\[
\vae = \|U_0- U^-_0\|_{2,\ga;(1+\gb);\D^-} +
\|b-b_0\|_{2,\ga;(\gb);(0,\infty)}^{*}.
\]

\smallskip
Now, for any function $f$ of $(U^-, U)$, we use
$\hat{f}$ to denote the value at the background states:
$\hat{f} = f(U^-_0, U^+_0)$. We also omit domain $Q^R$ and
boundary $\partial\mathcal{W}$ in the norms when no confusion arises. In the following estimates, all generic constants $C$'s and $c_0$ are  only dependent on the background states and $\alpha, \beta$, but independent of $\ve, \delta$.

\subsection{$C^0$ estimate for $v$}

We employ the comparison principles, Theorem \ref{thm-max1} and Theorem
\ref{thm-max2}, to estimate $v$.

We decompose matrix $\mathbf{A} = (\hat{a}_{ij})$ into $\mathbf{A}
= K K^\top$, where
\[K= (k_{ij}) =
\left(
\begin{array}{cc}
 \sqrt{\frac{\hat{a}_{11}\hat{a}_{22}- \hat{a}_{12}^2}{\hat{a}_{22}}}
  &\frac{\hat{a}_{12}}{\sqrt{\hat{a}_{22}}}\\[2mm]
  0 & \sqrt{\hat{a}_{22}} \\
\end{array}
\right),
\]
and $\hat{a}_{ij}, i,j=1,2$, are the coefficients of the second order terms in the
linearized equation of the nonlinear equation \eqref{eqn-Mbar}
at the background solution.

We define the transformation $\zz = K \bar{\zz}$, where
$\bar{\zz}=(\bar{z}_1, \bar{z}_2)$ is a new coordinate system. Then
$\sum_{j=1,2}k_{ij}\bar{z}_j = z_i$ implies
$$
\sum_{i,j=1,2} \hat{a}_{ij}\partial^2_{z_i z_j}=
\Delta_{\bar{\zz}}.
$$

We use the polar coordinates $(r, \theta)$ for $\bar{\zz}$ to
construct a comparison function for $v$. That is,
$$
r = |\bar{\zz}|, \qquad \theta =
\arctan\big(\frac{\bar{z}_2}{\bar{z}_1}\big).
$$
Let $\bar{\theta} = t \theta + \tau$.  Define
\begin{equation}\label{def-vbar}
 \bar{v}= r^s \sin
\bar{\theta},
\end{equation}
where $\tau>0$, and $s$ and $t$ will be chosen later.

We compute
\begin{eqnarray}
\sum_{i,j=1,2}\hat{a}_{ij}\partial^2_{z_iz_j} \bar{v}= \Delta_{\bar{\zz}}
\bar{v}=(s^2 - t^2)\, r^{s-2}  \sin \bar{\theta}.
\label{est-lap-vbar}
\end{eqnarray}

Let $s= -\beta$, $t = \ga$, and $0<\beta < \ga$ in \eqref{def-vbar}.
We set $v_1 = r^{-\gb} \sin (\ga \theta +\tau)$.

First, \eqref{4.11-a} leads to
\begin{align*}
A'(z_2) &= O(\ve) \max(z_2,1)^{-\gb -2} \min(z_2,1)^{\ga-1}\\
& = O(\ve) \max(z_2,1)^{-\ga- \gb -1 } [ \max(z_2,1) \min(z_2,1)]^{\ga-1}\\
& = O(\ve) \max(z_2,1)^{-\ga- \gb -1 } z_2^{\ga-1}\\
& = O(\ve)(z_2+1)^{-\ga- \gb -1 } r^{\ga-1} (\sin \theta)^{\ga-1} .
\end{align*}
 Since \eqref{4.12-a} gives $\|\varphi-\varphi^\infty\|_{2,\alpha;(\beta,0);Q}^{(-1-\alpha);\partial\mathcal{W}}\le \delta$ , by the expressions of $a^{\varphi}_{ij}, b^{\varphi}_j$ (see \eqref{def-aij}\eqref{def-bi} ), we have
\begin{align*}
|a^{\varphi}_{ij} - \hat{a}_{ij} | &\le C \delta.\\
(a^{\varphi}_{ij})_{z_j} &=  O(\delta) r^{-\gb-2} (\sin
\theta)^{\ga -1} +   O(1) A'(z_2) \\
&=  O(\delta) r^{-\gb-2} (\sin
\theta)^{\ga -1} +  O(\ve)(z_2+1)^{-\ga- \gb -1 } r^{\ga-1} (\sin \theta)^{\ga-1} .\\
b^{\varphi}_i &= O(\vae)r^{-\b -2}.\\
(b^{\varphi}_i)_{z_i} &= O(\vae)r^{-\b -3} + O(\ve)r^{-\b -2} A'(z_2) \\
&=  O(\vae)r^{-\b -3} +O(\ve)r^{-\b -2} (\sin \theta)^{\ga-1} .
\end{align*}

\begin{align*}
\hat{a}_{ij}\partial^2_{z_iz_j} v_1&=- (\ga^2 - \gb^2)\, r^{-\gb-2}  \sin \bar{\theta}.\\
(a^{\varphi}_{ij} - \hat{a}_{ij})\partial^2_{z_iz_j} v_1& = O(\delta) r^{-\gb-2} . \\
(a^\varphi_{ij}
)_{z_j} \partial_{z_i}v_1 &=  O(\delta) r^{-2\gb-3} (\sin
\theta)^{\ga -1} +  O(\ve)(z_2+1)^{-\ga- \gb -1 } r^{\ga-\gb-2} (\sin \theta)^{\ga-1} \\
&\le O(\delta) r^{-\gb-2} (\sin
\theta)^{\ga -1} +  O(\ve)z_2^{-\ga} r^{\ga-\gb-2} (\sin \theta)^{\ga-1} \\
&= O(\delta) r^{-\gb-2} (\sin
\theta)^{\ga -1} +  O(\ve) r^{-\gb-2} (\sin \theta)^{-1} .\\
b^{\varphi}_i\partial_{z_i}v_1 &=  O(\vae)r^{-2\b -3}.\\
( b^{\varphi}_i)_{z_i}v_1 & =   O(\vae)r^{-\b -2}(\sin \theta)^{\ga-1}.
\end{align*}
Therefore, we compute
\begin{align*}
L^\varphi v_1
&= (\hat{a}_{ij}\partial^2_{z_iz_j}
+ (a^\varphi_{ij} - \hat{a}_{ij})\partial^2_{z_i z_j})v_1+ (a^\varphi_{ij}
)_{z_j} \partial_{z_i}v_1 + b^\varphi_i\partial_{z_i}v_1
+ (b^\varphi_i)_{z_i} v_1\\
& \le -c_0 r^{-\gb-2}
+ O(\ve) \,r^{-2-\gb} (\sin \theta)^{-1},
\end{align*}
where $c_0 >0$ is a constant depending on the background state and $\ga,\gb$.

Let $v_2 = r^{-\gb} \sin^\ga \theta$. In the same manner, we can obtain the following estimate:
\begin{align*}
L^\varphi v_2& \le - c_1 r^{-\gb -2} \sin^{\ga -2} \theta.
\end{align*}

Set $v_3 =v_1 +v_2$. Thus, we have
\begin{eqnarray*}
  L^\varphi v_3 < -c_0 r^{-\gb-2} <0.
\end{eqnarray*}
By Theorem \ref{thm-max1}, $\frac{v}{v_3}$ achieves its positive
maximum on the boundary.

\smallskip
On $z_2 =0$ and $L^R$, $\frac{v}{v_3}\le C \ve$.

\smallskip
Let $\theta_0 = \arctan (-\frac{k_{11}}{k_{12}})$. We compute
$\nabla_{\zz} \bar{v}$  on the $z_2$-axis:
\begin{eqnarray*}
 &&\bar{v}_{\bar{z}_1} = r^{s-1}( s \cos \theta_0 \sin \bar{\theta}
  - t \sin \theta_0 \cos \bar{\theta}),\\[1mm]
 &&\bar{v}_{\bar{z}_2} = r^{s-1} (s \sin \theta_0 \sin \bar{\theta}
 + t \cos \theta_0 \cos \bar{\theta}).
\end{eqnarray*}
Then
\begin{eqnarray*}
&& \nabla_{\zz} \bar{v} =(K^{-1})^\top  \nabla_{\bar{\zz}} \bar{v}
  = r^{s-1} \left(
\begin{array}{c}
 \frac{s \cos \theta_0 \sin \bar{\theta}
 - t \sin \theta_0 \cos \bar{\theta} }{k_{11}}\\[2mm]
  \frac{s\sin \bar{\theta}}{k_{22} \sin \theta_0}
\end{array}
\right),\\[2mm]
&&(v_3)_{z_1} =-\frac{\gb \cos \theta_0 \sin \bar{\theta}
 + \ga \sin \theta_0 \cos \bar{\theta}
 + (\ga+\gb) \sin^\ga \theta_0 \cos \theta_0}{k_{11}} r^{-\gb -1},\\
&&(v_3)_{z_2} =-\frac{\gb (\sin \bar{\theta}+ \sin^\ga
\theta_0)}{k_{22} \sin \theta_0} r^{-\gb -1}.
\end{eqnarray*}
Then, when $\gb$ is suitably small, we have
$$
D_{\nnu} (v_1 + v_2)< -c_1 r^{-\gb -1}.
$$

Assume that $\frac{v}{v_3}$ achieves its maximum $\ve M$ at some
point $P$ on the $z_2$-axis. We know that
$D_\nnu(\frac{v}{v_3})(P) \le 0$.

Since $|g_0| = |\tilde{g}(\zz, \varphi_0, \nabla\varphi_0)-\tilde{g}
(\zz, \varphi^\infty, \nabla \varphi^\infty )|
\le C \ve r^{-\gb-1}$, we obtain that, at point $P$,
\begin{eqnarray}
 \nonumber 0 &\ge& D_\nnu v - \frac{v}{v_3} D_\nnu v_3\\
 \nonumber &=& g_0 - c^\varphi v - \ve M D_\nnu v_3\\
&\ge& - C\ve(1+ \ve M) r^{-\gb-1}  + M\ve c_1 r^{-\gb-1}.
\label{est-max-oblique}
\end{eqnarray}
This implies that $M \le \frac{2C}{c_1}$ for sufficiently small $\ve$.

Thus, we obtain the estimate for $v$:
\begin{equation}\label{est-v-beta}
    |v| \le C \ve r^{-\gb}.
\end{equation}

\smallskip
\subsection{$C^{1,\ga}$ estimate for $v$ at corner $O$}
In \eqref{def-vbar}, let $s = 1 + \ga$ and $t = 1+ \ga + \tau$. We define
\begin{equation}\label{def-v4}
v_4 = r^{1+ \ga} \sin \big((1+ \ga +\tau)\theta + \tau\big).
\end{equation}
By \eqref{est-lap-vbar}, it is easy to check
$$
L^\varphi v_4 <-c_2\, r^{\ga-1}.
$$
On the $z_2$-axis, we have
\begin{eqnarray*}
  D_\nnu v_4 &=& r^\ga\left(
   \frac{\nu_1}{k_{11}}\big((\ga +1) \sin((\ga +\tau)\theta_0 +\tau)
   - \tau \sin \theta_0 \cos \bar{\theta}\big)
   + \nu_2 \frac{(\ga +1)\sin \bar{\theta}}{k_{22} \sin
   \theta_0}\right)\\[1mm]
   &< & -c r^\ga,
\end{eqnarray*}
provided that $\ga$ and $\tau$ are suitably small.

Then we can use $\ve C v_4$ as a comparison function to control $w
\equiv v- v(\mathbf{0})- D_{\zz}v(\mathbf{0})\cdot \zz$ for $r <2$.

Denote any quarter ball  $B_r(\mathbf{0}) \cap Q$ with radius $r$ by $B^+_r$.
In $B^+_2$,
\begin{eqnarray}
  L^\varphi w  &=&\sum_{j=1,2}(f_j)_{z_j}
  := -\sum_{j=1,2}\big(\sum_{i=1,2}a_{ij}^\varphi v_{z_i}(\mathbf{0})
  + b_j^\varphi(v(\mathbf{0})+D_{\zz}v(\mathbf{0})\cdot \zz)\big)_{z_j} \nonumber
  \\
  &\ge& - C\ve r^{\ga-1} \ge L^\varphi (C\ve v_4). \label{def-bi-bar}
\end{eqnarray}
By Theorem \ref{thm-max2}, we have
\begin{equation*}
\sup_{B^+_2}\Big(\frac{w}{\ve C v_4}\Big)
\le \sup_{\partial B^+_2}\Big(\frac{w^+}{\ve Cv_4}\Big) +1.
\end{equation*}
On $\partial B^+_2 \cap (\{z_2=0\} \cup \{|\zz|=2\})$,
we see that
$\frac{w}{\ve C v_4} \le C$.

Assume that $\frac{w}{C\ve v_4}$ achieves its maximum M  at a point
$P$ on the $z_2$-axis. The oblique condition \eqref{con-obl-v}
implies
\begin{eqnarray*}
   \sum_{i=1,2}\nu_i^\varphi w_{z_i} + c^\varphi w =\bar{g}_0= {\rm O}(\ve r^\ga).
\end{eqnarray*}
The same argument as in \eqref{est-max-oblique} implies that, at the
maximum point $P$,
\begin{eqnarray*}
   0&\ge& D_{\nnu} (\frac{w}{ v_4} )\\[1mm]
   &=& \frac{1}{v_4}\big(D_\nnu w - \frac{w}{v_4} D_\nnu v_4\big)\\[1mm]
   &\ge& \frac{1}{v_4}\big(-c^\varphi w - \ve C r^\ga + \ve Mc_0 r^\ga\big),
\end{eqnarray*}
which implies that $M \le \frac{C}{c_0}$. Thus, $w\le \ve C r^{1+\ga}$ in
$B^+_2$.

Similarly, we obtain the corresponding lower bound.

Therefore,  we conclude
\begin{equation}\label{est-scale-w}
    |w(\zz)| \le \ve C r^{1+\ga} \qquad \mbox{for any} \,\,\,\zz \in B^+_2.
\end{equation}

With estimate \eqref{est-scale-w}, we can use the scaling technique
to obtain the $C^{1,\ga}$ estimate for $w$  up to the corner. More
precisely, for any point $P_\ast \in B_1^+$ with polar coordinates
$(d_\ast, \theta_\ast)$, we consider two cases for different values
of $\theta_\ast$.

\medskip
\textit{Case 1: Interior estimate for $\theta_\ast\in
[\frac{\pi}{6}, \frac{\pi}{3}]$}. Set $B_1 =
B_{\frac{d_\ast}{6}}(P_\ast)$ and $B_2 =
B_{\frac{d_\ast}{3}}(P_\ast)$. Then $B_1\subset B_2 \subset B_2^+$.
By the Schauder interior estimates ({\it cf.} (4.45) and  Theorem 8.33 in
\cite{gilbarg}), we have
\[
\|w\|^{(0)}_{1,\ga;B_2} \le C \Big(\|w\|_{0,0;B_2} +
\sum_{i=1,2}\|f_i\|^{(1)}_{0,\ga;B_2}\Big),
\]
where $f_i$ is defined in \eqref{def-bi-bar}, $C$ is a constant
independent of $d_\ast$, and the weight of the norm is up to $\po
B_2$. Therefore, by \eqref{est-scale-w}, we conclude
\begin{equation}
\|w\|_{1,\ga;B_1} \le d_\ast^{-(1+\ga)} \|w\|^{(0)}_{1,\ga;B_2} \le
C\ve.
\end{equation}

\textit{Case 2: Boundary estimate for $\theta_\ast >\frac{\pi}{3}$
or $\theta_\ast < \frac{\pi}{6}$}. Denote $B_3=Q \cap
B_{\frac{2d_\ast}{3}}(P_\ast)$. By the Schauder boundary
estimate ({\it cf.} (4.46) and Theorem 8.33 in \cite{gilbarg}), we have
\begin{eqnarray*}
&&  \|w\|_{1,\ga;B_3}^{(0)}\\
    &&\le C \Big(\|w\|_{0,0;B_3} +
\sum_{i=1,2}\|f_i\|^{(1)}_{0,\ga;B_3} +
\|\zeta\|_{1,\ga;\overline{B_3}\cap \{z_2=0\}}+
\|\bar{g}_0\|^{(1)}_{0,\ga;\overline{B_3}\cap \{z_1=0\}} \Big)\\
&&\le  \ve C d_\ast^{1+\ga}.
\end{eqnarray*}

Combining Case 1 with Case 2 yields the corner estimate:
\begin{equation}\label{est-corner}
    \|v\|_{1,\ga;B_1^+} = \|w+v(\mathbf{0})+ D_{\zz}v(\mathbf{0})\cdot \zz \|_{1,\ga;B_1^+}
     \le C\ve.
\end{equation}

\subsection{$C^{1,\ga}$ estimate for $v$ away from corner $O$} \label{c1ga}
Away from the corner, suppose  $\zz^* = (z_1^*, z_2^*) \in Q^R$ with polar coordinates
$(R^*, \theta^*)$ for $1 <R^* <\frac{R}{2}$. Two cases will be considered below.

\textit{Case 1.} $z^*_2 > \frac{1}{2}$.
Let
$B_\ast:=B_{\frac{z_2^*}{2}}(\zz^*) \cap Q$.
If $\theta^\ast> \arctan 2$, we employ
the Schauder boundary  estimate; for $\theta^\ast\le \arctan 2$, we employ the Schauder
 interior estimate. Notice that
 \[
 [a_{ij}^\varphi]^{(0)}_{0,\alpha; B_*} \le (z_2^* )^\alpha [a_{ij}^\varphi]_{0,\alpha; B_*} \le  C \ve \le 1.
 \]
 Therefore, $ |a_{ij}^\varphi|^{(0)}_{0,\alpha; B_*} \le \Lambda$, where $\Lambda >0$ is a constant only depending on the background states. Hence the constant $C$ in the following estimate is independent of $\zz^*$:
\begin{eqnarray*}
 \|v\|_{1,\ga;B_\ast}^{(0)}
    &\le& C \Big(\|v\|_{0,0;B_\ast} +
\|\bar{g}_0\|^{(1)}_{0,\ga;(\frac{z_2^*}{2}, \frac{3z_2^*}{2})} \Big)\\
&\le & C\ve (R^\ast)^{-\gb},
\end{eqnarray*}
where the weight in the superscript is the distance up to $(\partial B_{\frac{z_2^*}{2}}(\zz^*)) \cap Q$. Shrinking the domain by setting $B'_\ast:=B_{\frac{z_2^*}{4}}(\zz^*) \cap Q$  yields
\[ \|v\|_{1,\ga;(\gb,0);B'_\ast}
 \le C\ve .
\]

\textit{Case 2.} $z^*_2 \le \frac{1}{2}$.
 In this region, we apply Schauder boundary estimate in $B_1^+ := B_1(\zz^*) \bigcap Q$  to obtain
\begin{eqnarray*}
	\|v\|_{1,\ga;B_1^+}^{(0)}
	&\le& C \Big(\|v\|_{0,0;B_1^+} +
	\|\zeta\|_{1,\ga;(z_1^*-1,z_1^*+1)}
 \Big)\\
	&\le & C\ve (R^\ast)^{-\gb},
\end{eqnarray*}
which implies that
\[
\|v\|_{1,\ga; (\beta,0);B_{\frac{1}{2}}^+} \le C\ve .
\]
Combining the estimates in the two cases above, we obtain the  estimate for $v$ in $Q^{R/2}$
\begin{equation}\label{est-c1ga-beta-v}
    \|v\|_{1,\ga;(\gb,0); Q^{R/2}} \le C\ve.
\end{equation}

\subsection{$C^{2,\ga}$ regularity}
For the $C^{2,\ga}$ estimates with a weight to the $z_1$-axis, we
rewrite equation \eqref{eqn-v} into the non-divergence form:
\begin{equation}\label{eqn-v-nondiv}
   \sum_{i,j=1,2}\big( a^\vph_{ij} v_{z_i z_j}+ ((a^\vph_{ij})_{z_j} + b_i^\vph)v_{z_i}+
   (b^\vph_i)_{z_i}v\big)= 0,
\end{equation}
with the boundary condition on the $z_2$-axis:
\begin{equation}\label{con-bd-v-nondiv}
   \sum_{i,j=1,2}\hat{\nu}_i v_{z_i}= g_1,
\end{equation}
where
\begin{eqnarray*}
  &&g_1= g_0-c^\vph v + \sum_{i=1,2}(\hat{\nu}_i- \nu_i^\vph) v_{z_i}.
\end{eqnarray*}

We   use  the same scaling as in the $C^{1,\ga}$ estimates and consider two cases. One is away from the $z_1$-axis and the other is close to the $z_1$-axis.
We will estimate the $C^{2,\ga}$ norm in $ Q^{R/4}$. Suppose $\zz^\ast =(z_1^\ast, z_2^\ast) \in Q^{R/4}$.

\textit{Case 1.}  $z_2^* >1$.  $B_\ast:=B_{\frac{z_2^*}{2}}(\zz^*) \cap Q$ is the same as in the $C^{1,\ga}$ estimates. Set
\[
T= B_* \bigcap \{z_1 = 0\}.
\]
Similar to the situation in Section \ref{c1ga}, because the coefficients in \eqref{eqn-v-nondiv} have the proper decay rate in $z_2$ direction, we have
\[
|a_{ij}^\varphi|^{(0)}_{0,\alpha; B_*}, |(a^\vph_{ij})_{z_j} + b_i^\vph|^{(1)}_{0,\alpha; B_*},   |(b^\vph_i)_{z_i} |^{(2)}_{0,\alpha; B_*} \le \Lambda.
\]
The Schauder interior and boundary estimates ({\it cf.} Theorem 6.26 in
\cite{gilbarg}) imply
\begin{equation}\label{est-v-C2alpha-0B2}
  \|v\|^{(0)}_{2,\ga;B_*} \le C \left(\|v\|_{0,0;B_*}
   + \|g_1\|^{(1)}_{1,\ga;T}  \right).
\end{equation}
Then
\begin{eqnarray*}
   \|g_1\|^{(1)}_{1,\ga;T}&\le& C \big( |g_0|^{(1)}_{1,\ga;T}+
 \ve \|v\|^{(0)}_{2,\ga;B_*}\big)\\
 &\le& C \ve( |\zz^*|^{-\gb} + \|v\|^{(0)}_{2,\ga;B_*}  )
\end{eqnarray*}
leads to
\begin{equation}\label{est-vbeta}
 \|v\|_{2,\ga;(\gb,0);B'_*} \le C\ve .
\end{equation}

\textit{Case 2.}   $z^*_1 \le 1 $. Define $\bar{v}(\zz) =v(\zz)-v(z_1^*,0)- \nabla v(z_1^*,0) \cdot \zz  $ and let $B_k^+ = B_k(\zz^*) \cap Q$.
Then
\begin{eqnarray*}
	\| \bar{v}\|_{0,0;B_*}& \le& C [v]_{1,\ga;B_2^+} |z_2^*|^{1+\ga}\\
	&\le &  C \ve |\zz^*|^{- \gb}{z_2^*}^{1+\ga}.
\end{eqnarray*}
The same Schauder interior and boundary estimates as \eqref{est-v-C2alpha-0B2} applied to $\bar{v}$ gives rise to
\[
\|\bar{v}\|_{2,\ga;(0,0);B'_*}\le C\ve |\zz^*|^{-\gb} {z_2^*}^{1+\ga}.
\]
Hence, we conclude that
\begin{equation}\label{est-vbarcase2}
\|\bar{v}\|^{(-1-\ga);\partial \mathcal{W}}_{2,\ga;(\gb,0);B'_*}\le C\ve .
\end{equation}
Estimate \eqref{est-vbeta},\eqref{est-vbarcase2}, together with \eqref{est-v-beta}, leads to
\begin{equation}\label{est-v-c2al-QR}
\|v\|_{2,\ga;(\gb,0); Q^{R/4}}^{(-1-\ga);\partial \mathcal{W}} \le C\ve.
\end{equation}

Solution $v$ depends on $R$, which is denoted by $v^R$. By
compactness of $v^R$, we can find a subsequence converging to
$\tilde{v}$ such that
\begin{equation}
  \|\tilde{v}\|_{2,\ga;(\gb,0); Q}^{(-1-\ga);\partial \mathcal{W}} \le C\ve. \label{est-vtve}
\end{equation}
When $C \ve < \gd$, then $ \tilde{v} \in \Sigma_\gd$, and
$\tilde{v}$ is a solution of equation \eqref{eqn-v}.

\subsection{Uniqueness}
Because of the decay of $\tilde{v}$ at infinity, we can obtain the
uniqueness of $\tilde{v}$ by the comparison principle as follows:

Suppose that $v_1$ and $v_2$ are two solutions of \eqref{eqn-v}. The
difference $w= v_1 -v_2$ satisfies the same equation and
boundary conditions on the $z_2$-axis, and $w=0$ on the $z_1$-axis.

For any small positive constant $\tau$, we let $R$ to be large enough
such that $|w| \le \tau$ on the cutoff boundary $L^R$. Similar to
\eqref{est-max-oblique}, we employ Theorem \ref{thm-max1} to obtain
$|w| \le \tau $ in $Q^R$. Let $R \to \infty$ and $\tau \to \infty$,
we conclude that $w\equiv 0$, which implies the uniqueness.

\subsection{Determination of $\varphi$ as a fixed point}
 We define a map
$\mathcal{Q}: \Sigma_\gd \to  \Sigma_\gd$ by
$$
\mathcal{Q}(w) \equiv \tilde{v} \qquad\mbox{for any}\,\,\,\, w = \vph
-\vph^\infty,
$$
where the closed set $\Sigma_\gd$ is defined in \eqref{4.12-a}.
We employ the Schauder fixed point theorem to prove the existence
of a fixed point for $\mathcal{Q}$. That is, we need to verify the
following facts:

\begin{enumerate}
\item[(i)] $\Sigma_\gd$ is a compact and convex set in a Banach space
$\mathcal{B}$;

\smallskip
 \item[(ii)] $\mathcal{Q}: \Sigma_\gd \to  \Sigma_\gd $ is continuous in
 $\mathcal{B}$.
\end{enumerate}
Choose the Banach space $\mathcal{B}$
as defined in \eqref{4.13-a}.
Then $\Sigma_\gd$
is compact and convex in $\mathcal{B}$.

For the continuity of $\mathcal{Q}$, we make the following
contradiction argument. Let $w_0, w^n\in \Sigma_\gd$ and $w^n \to
w_0$ in $\mathcal{B}$. Then $v^n \equiv \mathcal{Q} (w^n)$ in
$\Sigma_\gd$ and  $v_0 \equiv \mathcal{Q} (w_0)$ in $\Sigma_\gd$. We
want to prove that $v^n \to v_0$ in $\mathcal{B}$.

Assume that $v^n \nrightarrow v_0$. Then there exist $c_0 >0$ and a
subsequence $\{v^{n_k}\}$ such that
$\|v^{n_k} - v_0\|_{\mathcal{B}} \ge c_0$.
Since $\{v^{n_k}\} \subset \Sigma_\gd$ is compact in
$\mathcal{B}$, we can find another subsequence, again denoted by
$\{v^{n_k}\}$, converging to some $v_1 \in \Sigma_\gd$.
Then $v_0$ and $v_1$ satisfy the same equation \eqref{eqn-v},
where $\vph= \vph^\infty + w_0$,
which contradicts with the uniqueness of
solutions for \eqref{eqn-v}.
Therefore, $\mathcal{Q}$ is continuous in
$\mathcal{B}$.

Thus, we have a fixed point $v$ for $\mathcal{Q}$, which gives a
solution $\vph \equiv \vph^\infty + v$ for the nonlinear equation
\eqref{eqn-Mbar}. The solution is unique by applying the same
comparison principle as for the linear equation.

Therefore, for given $(A, U^-, b)$, we have determined
$\varphi$.

\section{Determination of the Entropy Function $A$ as a Fixed Point}

In this section, we employ the implicit
function theorem to prove the existence of a fixed point $A$.

\subsection{Setup for the implicit function theorem for $A$}
Through the shock polar, we can determine the values of $U$ at $O$,
and hence $A(0)=A_t$ is fixed, depending on the values of $U_-(O)$
and $b'(0)$. Then we solve \eqref{con-Ht} to obtain a unique
solution $\tilde{A}= h (\zz , \varphi, \grad \varphi )$
that defines the iteration map.
To complete the proof, we need to prove that
the iteration map exists and has a fixed point
by the implicit function theorem.

In order to employ the implicit function theorem,
we need to set up a Banach space for $A$.
To realize this, we perform the following
normalization for $(A, U^-, b)$.

Let $A^+_0 = \frac{p_0^+}{(\rho_0^+)^\gc}$. Define a smooth cutoff
function $\chi$ on $[0,\infty)$ such that
\[
\chi(s)= \left\{%
\begin{array}{ll}
    1, & 0\le s <1, \\[1mm]
    0, & s>2. \\
\end{array}%
\right.
\]
Let $\omega = U^- - U^-_0$ and $\mu = b-b_0$. Set
\begin{equation}\label{def-fix-a0}
A(0) := t(\omega(0), \mu'(0)),
\end{equation}
 where $t$ is a function
determined by the Rankine-Hugoniot conditions
\eqref{con-RH1}--\eqref{con-RH4}.

\smallskip
Set $\gl = A - w_t$ with
\begin{equation}\label{6.2-a}
w_t(z_2) = A^+_0 + \big( t(\omega(0),
\mu'(0))-A^+_0\big)\chi(z_2).
\end{equation}
Then $\gl (0) =0$.

\medskip
Given $(\gl,\omega, \mu)$, we can compute $\tilde{A}-w_t=\mathcal{J}(A-w_t)$
that defines the iteration map, by constructing
a map $\tilde{\gl} = \tilde{A}-w_t \equiv \mathcal{P}
(\gl,\omega, \mu)$.

We will prove that equation $\mathcal{P}(\gl,\omega, \mu) - \gl
=0$ is solvable for $\gl$, given parameters $(\omega, \mu )$
near $(0,0)$. This is obtained by the implicit function theorem.

\subsection{Properties of the operator $\mathcal{P}$}
\label{sec-pro-oper}

We first define some Banach spaces for operator $\mathcal{P}$.
Set
\begin{equation}\label{def-space-x}
X= \{\gl\,: \, \gl(0) =0, \|\gl\|_X < \infty \}
\end{equation}
with
\begin{equation} \label{def-norm-x}\|\gl\|_X \equiv
\|\gl\|^{(-\ga);\{0\}}_{1,\ga;(1+\gb);(0,\infty)},
\end{equation}

\smallskip
\begin{equation}\label{def-space-y}
Y= \{\omega\,:\, \|\omega\|_Y < \infty \}
\end{equation}
with
\begin{equation}\label{def-norm-y}
\|\omega\|_Y \equiv \|\omega\|_{2,\ga;(\gb+1);\Omega_-}
\end{equation}
for a vector-valued function $\omega$, and
\begin{equation}\label{def-space-z}
Z= \{\mu\,:\, \mu(0) =0, \|\mu \|_Z < \infty \}
\end{equation}
with
\begin{equation}\label{def-norm-z}
\|\mu\|_Z \equiv \|\mu\|^{*}_{1,\ga; (\gb); (0, \infty)}.
\end{equation}

Clearly, $X,Y$, and $Z$ are Banach spaces. Operator
$\mathcal{P}$ is a map from $X\times Y\times Z$ to $X$.

We now define a linear operator $D_\gl \mathcal{P}(\gl,\omega, \mu)
$ and show that it is the partial differential of $\mathcal{P}$ with
respect to $\gl$.  When no confusion arises, we may drop the
variables $(\gl,\omega, \mu)$
in $D_\gl \mathcal{P}(\gl,\omega, \mu)$.

We divide the proof into four steps.

\medskip
{\bf 1}. {\it Definition of a linear operator
$D_\gl\mathcal{P}(\gl,\omega, \mu)$}. Given $\gd\gl\in X$, we solve
the following equation for $\gd\varphi$:
\begin{equation}\label{eqn-dphi}
  \sum_{i=1,2} \big(\sum_{j=1,2}a_{ij}^\gl (\gd \varphi)_{z_j} + b_i^\gl \gd \varphi
   + d_i^\gl \gd \gl\big)_{z_i}=0,
\end{equation}
with boundary conditions:
\begin{eqnarray}
&&\gd \varphi |_{z_2 =0} =0,\\[1mm]
&&\big(\sum_{i=1,2}\nu_i^\gl (\gd \varphi)_{z_i} + c^\gl \gd \varphi\big) |_{z_1 =0 }=0,
\end{eqnarray}
where
\begin{eqnarray*} &&a_{ij}^\gl= \overline{M}^i_{\varphi_{z_j}}
  (\zz, w_t+ \gl, \varphi, \nabla \varphi),\qquad
b_i^\gl= \overline{M}^i_{\varphi}
  (\zz, w_t+ \gl, \varphi, \nabla \varphi),\\
&&d_i^\gl= \overline{M}^i_A
  (\zz, w_t+ \gl, \varphi, \nabla \varphi),\quad
 \nu_i^\gl=g_{\varphi_{z_i}}(\zz,\varphi, \nabla \varphi),\quad
c^\gl=g_{\varphi}(\zz,\varphi, \nabla\varphi).
\end{eqnarray*}

Once we have known $\gd \varphi$, we define
\begin{equation}\label{def-tilde-delta-lambda}
    \widetilde{\gd \gl} = D_\gl \mathcal{P}(\gl,\omega,\mu) (\gd \gl)
    :=\sum_{i=1,2}e_i^\gl (\gd \varphi )_{z_i} + e_0^\gl \gd \varphi,
\end{equation}
where
\begin{eqnarray*}
 e_i^\gl=h_{\varphi_{z_i}}(\zz,\varphi, \nabla \varphi),\qquad
e_0^\gl=h_{\varphi}(\zz,\varphi, \nabla \varphi).
\end{eqnarray*}

It is easy to see that $\widetilde{\gd\gl}(0)=0$. Then
$D_\gl\mathcal{P}(\gl,\omega,\mu)$ is a linear operator from $X$ to
$X$.

\medskip
{\bf 2}. {\it Show that $D_\gl\mathcal{P}(\gl,\omega,\mu)$ is
the partial differential of $\mathcal{P}$ with respect to $\gl$ at
$(\gl,\omega,\mu)$}.

For fixed $(\omega,\mu)$, we need to estimate $\mathcal{P}(\gl + \gd
\gl,\omega,\mu)- \mathcal{P}(\gl,\omega,\mu)-D_\gl
\mathcal{P}(\gl,\omega,\mu)(\gd \gl)$ to be ${\rm o}(\gd \gl)$.

For $\gl$, we define $\varphi$ by following the definition of
$\mathcal{P}$, {\it i.e.}, we solve the following equation, an alternative
form from \eqref{eqn-Mbar}:
\begin{equation}\label{eqn-phi-lambda}
\sum_{i=1,2}\big(\overline{M}^i (\zz, w_t+ \gl, \varphi,
\grad\varphi)\big)_{z_i}=0,
\end{equation}
with  boundary conditions \eqref{con-phi-gt} and
\eqref{con-phi-z1}.

For $\gl + \gd \gl$, the corresponding potential $\bar{\varphi}$
satisfies
\begin{equation}\label{eqn-phibar-gdgl}
\sum_{i=1,2}\big(\overline{M}^i (\zz, w_t+ \gl + \gd \gl,
    \bar{\varphi}, \grad \bar{\varphi})\big)_{z_i}=0,
\end{equation}
with the same boundary conditions \eqref{con-phi-gt} and
\eqref{con-phi-z1}.

Taking the difference of equations \eqref{eqn-phi-lambda}
and \eqref{eqn-phibar-gdgl}  leads to the following equation:
\begin{equation}\label{eqn-phi-phibar}
\sum_{i=1,2}\big(\sum_{j=1,2}a_{ij}^{\gd \gl} (\bar{\varphi}- \varphi)_{z_j} + b_i^{\gd \gl}
(\bar{\varphi}-\varphi)+ d_i^{\gd \gl} \gd \gl\big)_{z_i}=0,
\end{equation}
with boundary conditions:
\begin{eqnarray}
  &&(\bar{\varphi}- \varphi) |_{z_2 =0} =0,\\[1mm]
 &&\big(\sum_{i=1,2}\nu_i^{\gd \gl} (\bar{\varphi}- \varphi)_{z_i} +
  c^{\gd \gl} (\bar{\varphi}- \varphi)\big)|_{z_1=0}=0,
\end{eqnarray}
where
\begin{eqnarray*}
&&a_{ij}^{\gd \gl}=\int_0^1 \overline{M}^i_{\varphi_{z_i}}
  (\zz, w_t+ \gl + s \gd \gl, \varphi + s(\bar{\varphi}- \varphi),
   \nabla (\varphi +s(\bar{\varphi}- \varphi)))\, {\rm d}s,\\[1mm]
&&b_i^{\gd \gl}=\int_0^1 \overline{M}^i_{\varphi}
  (\zz, w_t+ \gl + s \gd \gl, \varphi + s(\bar{\varphi}- \varphi),
   \nabla (\varphi +s(\bar{\varphi}- \varphi)))\, {\rm d}s,\\[1mm]
&&d_i^{\gd \gl}= \int_0^1 \overline{M}^i_A
  (\zz, w_t+ \gl + s \gd \gl, \varphi + s(\bar{\varphi}- \varphi),
   \nabla (\varphi +s(\bar{\varphi}- \varphi)))\, {\rm d}s,\\[1mm]
&& \nu_i^{\gd \gl}=\int_0^1 g_{\varphi_{z_i}}
 (\zz,\varphi + s(\bar{\varphi}- \varphi),
   \nabla (\varphi +s(\bar{\varphi}- \varphi)))\, {\rm d}s,\\[1mm]
&&c^{\gd \gl}=\int_0^1 g_{\varphi}
 (\zz,\varphi + s(\bar{\varphi}- \varphi),
   \nabla (\varphi +s(\bar{\varphi}- \varphi)))\, {\rm d}s.
\end{eqnarray*}

Take the difference of \eqref{eqn-phi-phibar} and \eqref{eqn-dphi},
and let $v := \bar{\varphi}- \varphi - \gd \varphi$. Then we
have
\begin{eqnarray}\nonumber
 &&\sum_{i=1,2}\big(\sum_{j=1,2}a_{ij}^{\gd \gl} v_{z_j} + b_i^{\gd \gl}
    v\big)_{z_i}\\
 &&=  -\sum_{i=1,2}\big(
 \sum_{j=1,2}(a_{ij}^{\gd \gl}-a_{ij}^{ \gl}) (\gd \varphi)_{z_j}+
  (b_i^{\gd \gl}-b_i^{\gl})\gd \varphi +
   (d_i^{\gd \gl}-d_i^{\gl}) \gd \gl \big)_{z_i}\nonumber\\
&&\equiv E^{\gd \gl},\label{eqn-dphibar-dphi}
\end{eqnarray}
with boundary conditions:
\begin{eqnarray*}
  &&v |_{z_2 =0} =0,\\[1mm]
 &&\sum_{i=1,2} \nu_i^{\gd \gl} v_{z_i} +
  c^{\gd \gl} v
  =\sum_{i=1,2} (\nu_i^{ \gl}-\nu_i^{\gd \gl})(\gd \varphi)_{z_i}
  + (c^{\gl}-c^{\gd \gl})\gd \varphi \qquad \mbox{for}\quad  z_1=0.
\end{eqnarray*}

Since $\bar{\varphi}- \varphi = {\rm O}(\gd \gl)$ and
$\gd \varphi= {\rm O}(\gd \gl)$, we conclude that
$E^{\gd \gl}$ and $\sum_{i=1,2}(\nu_i^{ \gl}-\nu_i^{\gd \gl})(\gd
\varphi)_{z_i} + (c^{\gl}-c^{\gd \gl})\gd \varphi$
are ${\rm o}(\gd \gl)$.
Thus, we have
$$
v={\rm o}(\gd \gl).
$$
Then
\begin{eqnarray*}
&&\mathcal{P}(\gl + \gd \gl,\omega,\mu)-
\mathcal{P}(\gl,\omega,\mu)-D_\gl \mathcal{P}(\gl,\omega,\mu)(\gd
\gl)\\[1mm]
  &&:=\widetilde{\bar{\gl}} - \tilde{\gl} - \widetilde{\gd \gl}\\[1mm]
  &&= h(\zz, \bar{\phi}, \nabla \bar{\phi}) - h(\zz, {\phi}, \nabla
  {\phi})-\sum_{i=1,2} e_i^\gl (\gd \varphi )_{z_i} - e_0^\gl \gd \varphi\\[1mm]
  &&=\sum_{i=1,2}e_i^{\gd \gl}v_{z_i} + e_0^{\gd \gl} v +
\sum_{i=1,2}(e_i^{\gd \gl}-e_i^\gl) (\gd \varphi )_{z_i}+(e_0^{\gd \gl}  -
e_0^\gl) \gd \varphi,
\end{eqnarray*}
where
\begin{eqnarray*}
&&e_i^{\gd \gl}= \int_0^1 h_{\varphi_{z_i}}
 (\zz,\varphi + s(\bar{\varphi}- \varphi),
   \nabla (\varphi +s(\bar{\varphi}- \varphi)))\, {\rm d}s,\\
&&e_0^{\gd \gl}= \int_0^1 h_{\varphi}
 (\zz,\varphi + s(\bar{\varphi}- \varphi),
   \nabla (\varphi +s(\bar{\varphi}- \varphi)))\, {\rm d}s.
\end{eqnarray*}
Therefore, we conclude that $\widetilde{\bar{\gl}} - \tilde{\gl} -
\widetilde{\gd \gl} = {\rm o}(\gd \gl)$. Thus, $D_\gl
\mathcal{P}(\gl,\omega,\mu)$ is {\it the partial differential} of
$\mathcal{P} $ with respect to  $\gl$ at $(\gl,\omega,\mu)$.

\medskip
{\bf 3.} {\it Continuity of $\mathcal{P}$ and $D_\gl \mathcal{P}$}.
It suffices to show the continuity of $\mathcal{P}$ at any point
$(\gl^\ast,\omega^\ast,\mu^\ast)$ near $(0,0,0)$ by a contradiction
argument, since the same argument applies to $D_\gl \mathcal{P}$.

Assume that there exists a sequence $ (\gl^n,\omega^n,\mu^n) \to
(\gl^\ast,\omega^\ast,\mu^\ast)$ in $X \times Y \times Z$, while
$\|\widetilde{\gl^n}- \widetilde{\gl^\ast}\|_X \ge c_0 >0$.
Using the compactness of $\widetilde{\gl^n}$ in $\|\cdot\|_{\ga'}$  and
the compactness of $\varphi^n$ in $\| \cdot \|_{1,\ga'}$ with
$\ga'<\ga$, we find a subsequence $\{n_k\}$ such that
$\widetilde{\gl^{n_k}}$ converges to some
$\widetilde{\gl^{\ast\ast}}$ in $\|\cdot \|_{\ga'}$
and $\varphi^n$
converges to some $\varphi^{\ast\ast}$ in $\|\cdot \|_{1,\ga'}$.
Now we see that $\varphi^{\ast}$ and $\varphi^{\ast\ast}$ satisfy
the same equation \eqref{eqn-Mbar} with the same boundary
conditions. By the uniqueness of solutions for \eqref{eqn-Mbar},
we conclude that $\varphi^{\ast}=\varphi^{\ast\ast}$.
This implies that
$\widetilde{\gl^{\ast}} = \widetilde{\gl^{\ast\ast}}$.
However, by
assumption,
$\|\widetilde{\gl^{\ast\ast}}-\widetilde{\gl^\ast}\|_X \ge c_0 >0$.
This leads to a contradiction.
Therefore, $\mathcal{P}$ is continuous.

\medskip
{\bf 4.} {\it Show that, at the background state
$(\gl,\omega,\mu)= (0,0,0)$, $D_\gl \mathcal{P}(0,0,0) -I$ is
an isomorphism.}

When $(\gl,\omega,\mu)= (0,0,0)$, we solve for $\gd \varphi$:
\begin{equation}\label{eqn-gdphi}
\sum_{i=1,2}\big(\hat{M}^i_A \gd \gl +  \hat{M}^i_j (\gd \varphi)_{z_j}\big)_{z_i}
= 0   \qquad\mbox{in $Q$},
\end{equation}
with the boundary conditions:
\begin{eqnarray}
 &&\sum_{i=1,2} \hat{g}_i (\gd \varphi)_{z_i} |_{z_1 =0} = 0,
  \label{con-bdz2-gd-phi}\\[2mm]
 && \gd \varphi |_{z_2 =0} = 0, \label{con-bdz1-gd-phi}
\end{eqnarray}
where $(\hat{M}^i_A, \hat{M}^i_j, \hat{g}_i)$ are the
corresponding
$(\overline{M}_A, \overline{M}^i_{\varphi_{z_j}}, \tilde{g}_{\varphi_{z_i}})$
evaluated at the background state
$(U_-^0, U_+^0, \zeta_0)$.

Then we have
$$
D_\gl \mathcal{P} (\gd \gl)
:=\widetilde{\gd \gl}
= \sum_{i=1,2}\hat{h}_i(\gd \varphi)_{z_i},
$$
where $\hat{h}_i:=h_{\varphi_{z_i}}$ evaluated at the background state.

We rewrite the system in the following way:
Let
\[
m = \frac{\hat{M}^2_A}{\hat{M}^2_2}, \qquad
  \overline{\gd \varphi} = \gd \varphi +
  m \int_0^{z_2} \gd \gl (s)\, ds.
\]
Then \eqref{eqn-gdphi}--\eqref{con-bdz1-gd-phi} become
\begin{eqnarray}
  &&\sum_{i,j=1,2}\hat{M}^i_j (\overline{\gd \varphi})_{z_i z_j} = 0 \qquad
    \mbox{in} \,\, Q, \label{eqn-gdphibar}\\[1mm]
  &&\sum_{i=1,2}\hat{g}_i (\overline{\gd \varphi})_{z_i }|_{z_1=0}  \label{con-ghat-bar}
   = m \hat{g}_2 \gd \gl,
 \\[1mm]
&&\overline{\gd \varphi}|_{z_2=0}= 0. \label{con-phi=0}
\end{eqnarray}
Then
\begin{equation}\label{eqn-gdAbar}
    \widetilde{\gd \gl } =\sum_{i=1,2} \hat{h}_i  (\overline{\gd \varphi})_{z_i
    }(0,z_2)
    -\hat{h}_2 m \gd \gl.
\end{equation}
Equations \eqref{con-ghat-bar} and \eqref{eqn-gdAbar} give rise to
\begin{eqnarray}
 ( \hat{g}_2 \hat{h}_1 - \hat{g}_1 \hat{h}_2)
 (\overline{\gd \varphi})_{z_1}=\hat{g}_2 \widetilde{\gd \gl}.
 \label{con-phi1}
\end{eqnarray}
Noticing that $\widetilde{\gd \gl} (0)=0$ and $\hat{g}_2 \hat{h}_1 -
\hat{g}_1 \hat{h}_2 \ne 0$, the boundary conditions \eqref{con-phi1}
and \eqref{con-phi=0} are compatible to guarantee the unique
solution of \eqref{eqn-gdphibar} for arbitrary $\widetilde{\gd \gl}$.
This implies that $I - D_\gl \mathcal{P}$ is onto.

When $\gd \gl-D_\gl \mathcal{P} (\gd \gl) = 0$, \eqref{eqn-gdAbar}
becomes
\begin{equation}\label{eqn-gdAbar1}
    \sum_{i=1,2} \hat{h}_i  (\overline{\gd \varphi})_{z_i
    }(0,z_2)
   = (1+\hat{h}_2 m) \gd \gl.
\end{equation}
Canceling $\gd \gl$ in \eqref{con-ghat-bar} and
\eqref{eqn-gdAbar1} implies
\begin{equation}\label{con-phibar}
\big(\hat{g}_1 + m( \hat{g}_1 \hat{h}_2 -\hat{g}_2 \hat{h}_1)\big)
(\overline{\gd \varphi})_{z_1 }  + \hat{g}_2(\overline{\gd
\varphi})_{z_2}=0.
\end{equation}

The oblique boundary condition \eqref{con-phibar} above is nondegenerate,
since $\hat{g}_2 \ne 0$.
Therefore, solving equation \eqref{eqn-gdphibar} with boundary conditions  \eqref{con-phi=0}
and \eqref{con-phibar} leads to  $\overline{\gd \varphi} \equiv 0$. Using \eqref{con-ghat-bar},
we conclude that
$\gd \gl \equiv 0$,  which implies that $I - D_\gl\mathcal{P}$  is one-to-one.
Thus, $I - D_\gl \mathcal{P}$ is an isomorphism.

Therefore, given $U^-$ and $b$, operator $\mathcal{P}$ has a
fixed point $\gl$, which determines $A=w_t+\gl$.

\smallskip
With $A$ from $\gl$, we obtain a unique potential $\varphi$. Using the Hodograph transformation \eqref{def-hodo}, we can solve for $\phi$, with estimate
\[
\|\phi-\phi^\infty\|^{(-1-\alpha); \partial \mathcal{W}}_{2,\alpha; (\beta,0); \D_{\hat{\sigma}}}
\le C \ve.
\]
The subsonic flow $U$ can be expressed by a smooth function of  $(\grad \phi, A, B)$,
 i.e. $ U(\yy) = F(A(y_2), B(y_2), \grad \phi(\yy))$.
 Also, $U^\infty = F(A, B, \grad \phi^\infty) $. Then
 \begin{align*}
 U- U^\infty &= \int_0^1 F_{\grad\phi }(A,B, \grad \phi^\infty + s(\grad \phi - \grad \phi^\infty)) ds(\grad \phi - \grad \phi^\infty) \\
 & \triangleq G(A,B, \grad v)  \grad v.
 \end{align*}
Based on the expression above, we can conclude
\[
\|U - U^\infty  \|^{(-\alpha); \partial \mathcal{W}}_{1,\alpha; (\beta,1); \D_{\hat{\sigma}}} \le C\ve.
\]
Since the Lagrangian coordinate transformation is bi-Lipschitz, we can change the $\yy$-coordinates back to  the $\xx$-coordinates and
 complete the existence part of Theorem 2.1.

\section{Decay of the Solution to the Asymptotic State in the Physical Coordinates}

Now we determine the decay of the solution to the asymptotic state $U^\infty$
in the $\xx$-coordinates. We divide the proof into four steps.

\medskip
{\bf 1.} For the fixed point established in \S 5.7, estimate \eqref{est-vtve} implies
$$
\varphi-\varphi^\infty\in \Sigma_{C\ve}.
$$
Then the change of variables from the $\zz$-coordinates to $\yy$-coordinates yields
\begin{equation}\label{P-1}
\|\hat{\sigma}-\hat{\sigma}_0\|^{*, (-1-\alpha); \{0\}}_{2,\alpha; (\beta);\R^+}
+\|\phi-\phi^\infty\|^{(-1-\alpha); \partial \mathcal{W}}_{2,\alpha; (\beta,0); \D_{\hat{\sigma}}}
\le C \ve,
\end{equation}
where $\D_{\hat{\sigma}}$ is the subsonic region defined in $\eqref{3.12a}$.

\smallskip
{\bf 2.}  From \eqref{def-norm-x}, \eqref{P-1}, and Step 4 in \S 6.2, we have
\begin{equation}\label{P-2}
\|A^\infty-A_0^+\|_{1,\alpha; (1+\beta); \R^+}^{(-\alpha); \{0\}}
\le C\ve.
\end{equation}
Then, from \S 5.1, we obtain that, for $U^\infty=U^\infty (y_2)=(\uu^\infty, p_0^+, \rho^\infty)(y_2)$,
\begin{equation}\label{P-2a}
\|U^\infty-U_0^+\|^{(-\ga);\{0\}}_{1, \alpha; (1+\beta); \R^+}\le C\ve,
\end{equation}
and
\begin{equation}
\uu^\infty\cdot (\sin\theta_0, -\cos\theta_0)=0.
\end{equation}

\smallskip
{\bf 3.} Since $x_2=\phi(\yy)$, we now estimate $\phi(\yy)-\tan\theta_0\,y_1$, which is
$x_2-\tan \theta_0\, x_1$.
From \S 5.1 and \eqref{P-1},
$$
\phi^\infty-\tan \theta_0\, y_1=l(y_2),
$$
so that
\begin{equation}\label{P-3}
\|(\phi-\tan\theta_0\,y_1)-l(y_2)\|_{2,\alpha; (\beta,0); \D_{\hat{\sigma}}}^{(-1-\alpha); \partial\mathcal{W}}\le C\ve.
\end{equation}
In particular, this implies that, for each $y_2>0$,
\begin{equation}\label{P-4}
l'(y_2)=\lim_{y_1\to \infty}\partial_{y_2}(\phi(\yy)-\tan\theta_0\,y_1)
=\lim_{y_1\to\infty}\phi_{y_2}(\yy).
\end{equation}
By \eqref{eqn-phi-deri},
$\phi_{y_2}=\frac{1}{\rho u_1}$. By \eqref{P-2},
$$
\phi_{y_2}\ge \frac{1}{\rho_0^+ u_{10}^+}-C\ve \ge \frac{1}{2\rho_0^+ u_{10}^+}
\qquad\mbox{for small $\ve>0$}.
$$
Therefore, there exists $C>0$ such that, for any $y_2>0$,
$$
\frac{1}{C}\le l'(y_2)\le C.
$$
Since $\psi(\mathbf{0})=0$, we conclude that $l(0)=w_0$,
where $w_0$ is the limit of $b - b_0$ (see \S \ref{Section-5.1}).

Furthermore, by \eqref{P-3}--\eqref{P-4}, and \eqref{5.1a} with
\eqref{P-2}, we have
\begin{equation}\label{P-5}
\big\|l-\frac{y_2}{\rho_0^+ u_{10}^+}\big\|_{2,\alpha; (\beta); \R^+}^{*, (-1-\alpha); \{0\}}\le C\ve.
\end{equation}
Then there exists $g: [w_0,\infty)\to [0,\infty)$ with $g(w_0)=0$ such that
$g=l^{-1}$:
$$
g(l(y_2))=y_2 \qquad \mbox{on $(0,\infty)$}
$$
and
$$
\frac{1}{C}\le g'(s)\le C,
$$
so that $g(\cdot)$ satisfies \eqref{P-5}. In order to composite $g$ with $\phi-\tan\theta_0\,y_1$, we need to extend $g$ to a larger  domain. Notice that $\phi_{y_2 } \ge \frac{1}{2\rho_0^+ u_{10}^+}$ implies that $\phi$ is an increasing function in $y_2$. Therefore, condition \eqref{Con-bd-phi-wall} and the assumption \eqref{con-wallTS} on $b$ yields \[\phi-\tan\theta_0\,y_1 \ge -\ve.
\]
Thus, we will extend $l$ to domain $[-w_1, \infty)$ so that $g$ is defined on $[-\ve, \infty)$.
 By  \eqref{P-5}, we know that $l'(y_2) \ge \frac{1}{\rho_0^+ u_{10}^+}-C\ve \ge \frac{1}{2\rho_0^+ u_{10}^+} $, for small $\ve$. Take $w_1 =  4\rho_0^+ u_{10}^+ \ve$ and we use the same extension map as in Lemma 6.37 in \cite{gilbarg} to extend $l$ to $[-w_1, \infty)$ as follows. The extended function is still denoted by $l$. Let
 \[
 h(s) = l(s)-\frac{s}{\rho_0^+ u_{10}^+}.
 \]
 For $s \in [-w_1, 0)$, define
\[
 h(s) = \sum_{i=1}^2 c_i h(-s/i),
 \]
 where $c_1, c_2$ are constants determined by the algebraic equations
 \[
  \sum_{i=1}^2 c_i (-1/i)^m = 0, \qquad m=0,1.
 \]
 Then define
 \[
 l(s) = h(s) + \frac{s}{\rho_0^+ u_{10}^+},
 \]
 for $s \in [-w_1, 0) $. The extended function $l$ satisfies
 \begin{align*}
 \big\|l-\frac{y_2}{\rho_0^+ u_{10}^+}\big\|_{1,\alpha; (\beta); [-w_1, \infty)}^{*} \le C\ve,\\
l'(y_2)  \ge \frac{1}{2\rho_0^+ u_{10}^+} , \quad \text{ for }  y_2 \in [-w_1 ,\infty).
 \end{align*}
 Let $- w_2 = l(-w_1)$ and we know
 \begin{align*}
 w_2 & = -l(0) + \int_{-w_1}^{0} l'(s) ds\\
 & \ge -w_0 + \frac{w_1}{2\rho_0^+ u_{10}^+}\\
 & \ge \ve \quad (|w_0| \le \ve \text{ by } \eqref{con-wallTS}).
 \end{align*}
 Thus, $g = l^{-1}$ is defined on $[-\ve, \infty)$.
Therefore, by \eqref{P-3}, we have
$$
\|g(\phi(\yy)-\tan\theta_0\,y_1)-y_2\|_{1,\alpha; (\beta,0); \D_{\hat{\sigma}}}\le C\ve.
$$
In the same manner as for $l$, we can extend $U^\infty$ to domain $[-w_1, \infty)$ such that
\[\|U^\infty-U_0^+\|_{0, \alpha; (1+\beta);[-w_1, \infty)}\le C\ve.
\]
We define
$$
V^\infty(s)=U^\infty(g(s)),   \quad \text{ for }  s \in [-\ve ,\infty) .
$$
Then we employ \eqref{P-2a} to obtain
\begin{equation}\label{P-6}
\|V^\infty(\phi(\yy)-\tan\theta_0\,y_1)-U^\infty(y_2)\|_{0,\alpha; (\beta,1); \D_{\hat{\sigma}}}\le C\ve.
\end{equation}

\smallskip
{\bf 4.} Next, we use that
the change of variables $\yy\to \xx$ is globally bi-Lipschitz, which follows from \eqref{P-1} and
\eqref{eqn-phi-deri} that implies the Jacobian:
$$
J=\frac{1}{\rho^\infty u_1^\infty}\ge \frac{1}{C}>0
$$
if $\ve$ is small, by \eqref{P-2a}.

We also note that, in the $\yy$-coordinates, \eqref{P-1} implies for $U^\infty=U^\infty(y_2)$
that
$$
\|U-U^\infty\|_{1,\alpha; (\beta,1); \D_{\hat{\sigma}}}^{(-\alpha); \partial\mathcal{W}}\le C\ve.
$$
Then, changing the variables from $\yy$ to $\xx$ (which is bi-Lipschitz)
and using \eqref{P-6}, we obtain that, in the $\xx$-coordinates with $x_2=\phi(\yy)$,
$$
\|U-V^\infty(x_2-\tan\theta_0\,x_1)\|_{0,\alpha; (\beta,1); \Omega_{\mathcal{S}}}\le C\ve.
$$
This completes the proof for the decay of solution $U(\xx)$ to the asymptotic state $U^\infty$.

\section{Stability of Solutions}

In this section, we prove that the subsonic solutions are stable
under small perturbations of the incoming flows and the wedges
as stated in Theorem 2.1. We modify operator $\mathcal{P}$ into
$\overline{\mathcal{P}}$ as follows:

We first modify the definitions of the spaces in
\eqref{def-norm-x}--\eqref{def-space-z} in \S \ref{sec-pro-oper}
by discarding the constraints:
\begin{eqnarray*}
 \overline{X}= \{\gl :
\|\gl\|_X < \infty \}, \qquad
  \overline{Z}=\{\mu : \|\mu\|_Z < \infty \} ,
\end{eqnarray*}
where norms $\| \cdot \|_X$ and $\| \cdot \|_Z$ are  defined
in \eqref{def-norm-x} and \eqref{def-norm-z}, respectively. We
still use the same space $Y$ and the related norm as in
\eqref{def-space-y} and \eqref{def-norm-y}.

Let $\omega = U^- - U^-_0$, $\gl = A - A^+_0$, and $\mu = b-b_0$.
Given $(\gl, \omega, \mu)$, we define
$\overline{\mathcal{P}}(\gl, \omega, \mu)$ in the same way
as for $\mathcal{P}$ by the end of
\S \ref{sec-ell-pro}, except that we do not restrict the value
of $A(0)$ by \eqref{def-fix-a0}.
The restriction for $A(0)$ is
essential for the isomorphism of $D_\gl \mathcal{P}$.
To prove the stability, we need to eliminate this restriction
so that the differentiability in $\omega$ can be achieved
in a larger space.

Equation \eqref{eqn-Mbar} can be written as
\begin{equation}\label{eqn-mbar-U-}
    \sum_{i=1,2} \overline{M}^i(U_-(\varphi, z_2), A_0^+ + \gl, \nabla
    \varphi)_{z_i}=0.
\end{equation}

Given $(\gd \gl, \gd \omega, \gd \mu) \in \overline{X} \times Y
\times \overline{Z}$, define $D\overline{\mathcal{P}}(\gd \gl, \gd
\omega, \gd \mu) $ in the following way: We solve the following
equation for $\gd \varphi$:
\begin{equation}\label{eqn-dphi-a}
\sum_{i=1,2}\big(\sum_{j=1,2}a_{ij}^\gl (\gd \varphi)_{z_j} + b_i^\gl \gd \varphi
   + d_i^\gl \gd \gl + f_i^\gl \cdot \gd \omega\big)_{z_i}=0
\end{equation}
with the boundary conditions:
\begin{eqnarray}
  &&\gd \varphi |_{z_2 =0} =-\frac{\gd \mu(\tilde{b})
  - \gd \phi^-(\tilde{b}, 0)}{b' (\tilde{b})-
  (\phi^-)_{y_1}(\tilde{b},0)},\\[1mm]
 && \big(\sum_{i=1,2}\nu_i^\gl (\gd \varphi)_{z_i} + c^\gl \gd \varphi
  + w^\gl\cdot \gd \omega\big)|_{z_1 =0 }=0,
\end{eqnarray}
where
\begin{eqnarray*}
&&a_{ij}^\gl=\overline{M}^i_{\varphi_{z_i}}
  (U^-(\varphi, z_2), A_0^+ + \gl, \varphi, \nabla \varphi),\\[1mm]
&&b_i^\gl= \overline{M}^i_{U^-}
  (U^-(\varphi, z_2), A_0^+ + \gl, \varphi, \nabla \varphi)
  \cdot (U^-)_{y_1} (\varphi, z_2),\\[1mm]
&&d_i^\gl=\overline{M}^i_A
  (U^-(\varphi, z_2), A_0^+ + \gl, \varphi, \nabla \varphi),\\[1mm]
&&f_i^\gl= \overline{M}^i_{U^-}
  (U^-(\varphi, z_2), A_0^+ + \gl, \varphi, \nabla \varphi),\\[1mm]
&&\nu_i^\gl=g_{\varphi_{z_i}}(U^-(\varphi, z_2), \nabla \varphi),\\
&&c^\gl= g_{U^-}(U^-(\varphi, z_2), \nabla\varphi)\cdot(U^-)_{y_1}(\varphi, z_2),\\[1mm]
&&w^\gl = g_{U^-}(U^-(\varphi, z_2), \nabla \varphi).
\end{eqnarray*}
Then we define
\begin{equation}\label{def-tilde-delta-lambda1}
    \widetilde{\gd \gl} = D \overline{\mathcal{P}}(\gl,\omega,\mu) (\gd \gl, \gl\omega,\gl \mu)
    :=\sum_{i=1,2}e_i^\gl (\gd \varphi )_{z_i} + e_0^\gl \gd \varphi +  w_1^\gl\cdot \gd \omega,
\end{equation}
where
\begin{eqnarray*}
&&e_i^\gl= h_{\varphi_{z_i}}(U^-(\varphi, z_2), \nabla \varphi  ),\\[1mm]
&&e_0^\gl= h_{U^-}(U^-(\varphi, z_2), \nabla\varphi)\cdot(U^-)_{y_1}(\varphi, z_2),\\[1mm]
&&w_1^\gl = h_{U^-}(U^-(\varphi, z_2), \nabla \varphi).
\end{eqnarray*}

Following the same estimates as in \S \ref{sec-pro-oper}, we
can verify that $D \overline{\mathcal{P}}$ is the differential of
$\overline{\mathcal{P}}$.

Let $\mathcal{R}:=\overline{\mathcal{P}}-I$. In \S \ref{sec-pro-oper},
given $(\omega, \mu)$, we can find the fixed point
for $A$. Therefore, $\gl$ is a function of $(\omega, \mu)$, denoted
by $\gl(\omega, \mu)$. Therefore, we have
\begin{equation}\label{eqn-pbar}
  \mathcal{R} (\gl(\omega, \mu), \omega, \mu) =0.
\end{equation}

\medskip
Suppose that there is another parameter
$(\bar{\omega}, \bar{\mu})$ so that
\begin{equation}\label{eqn-pbar-bar}
  \mathcal{R} (\gl(\bar{\omega}, \bar{\mu}), \bar{\omega}, \bar{\mu}) =0.
\end{equation}

Taking the difference of equations \eqref{eqn-pbar} and
\eqref{eqn-pbar-bar}, we have
\begin{eqnarray*}
D_\gl \mathcal{R}(\gl,\omega,\mu) (\bar{\gl}-\gl) +
D_{(\omega,\mu)}(\bar{\omega} - \omega, \bar{\mu}-\mu)
+ {\rm o}(\bar{\gl}-\gl,\bar{\omega} - \omega, \bar{\mu}-\mu) =0,
\end{eqnarray*}
where $\bar{\gl} = \gl(\omega, \mu)$.
Since $D_\gl \mathcal{R}(\gl,\omega,\mu)$ is an isomorphism near the background
state, by inverting $D_\gl \mathcal{R}$, we obtain
\begin{eqnarray*}
 \bar{\gl} - \gl=-(D_\gl \mathcal{R})^{-1} D_{(\omega,\mu)}
 (\bar{\omega} - \omega, \bar{\mu} - \mu) + {\rm o}(\bar{\gl} -
  \gl,\bar{\omega} - \omega, \bar{\mu} - \mu).
\end{eqnarray*}
Therefore, we obtain the following inequality:
\begin{equation}\label{est-stab}
\|\bar{\gl} - \gl\|_X \le C\big(\|\bar{\omega} - \omega \|_Y +
\|\bar{\mu} - \mu \|_Z\big),
\end{equation}
which implies the stability of the solutions depending on the
perturbation of both the incoming flows and wedge boundaries.

\section{Remarks on the Transonic Shock Problem when $U^+_0$ Is on Arc $\wideparen{TH}$}
In this case, $\nu^\varphi_1 >0$ and
$\nu^\varphi_2>0$ in the boundary condition \eqref{con-obl-v}, which
makes a significant difference from the case when $U^+_0$ is on
arc $\wideparen{TS}$. Such a difference may affect the estimates,
hence the smoothness of the solutions, in general.

In particular, one may not expect a solution for case
$\wideparen{TS}$ is $C^{1, \alpha}$; it is generically only in
$C^\alpha$.

For example, in the first quadrant, let domain $OAB$ be the
quarter of the unit disc.
The oblique direction $\nnu = (-1,-1)$.
Let $u = r^{\frac{1}{2}} \sin(\frac{\theta}{2})$,
where $(r, \theta)$ are the polar coordinates.
In $OAB$, $u$ satisfies the Laplace equation:
\begin{equation}\label{eqn-laplace-u}
    \Delta u =0,
\end{equation}
and the boundary conditions:
$u=0$ on $OA$, and $\nabla u \cdot \nnu = 0$ on $OB$.
However, $u$ is  H\"older continuous only in $C^{\frac{1}{2}}$.

\begin{figure}
 \centering
\includegraphics[height=55mm]{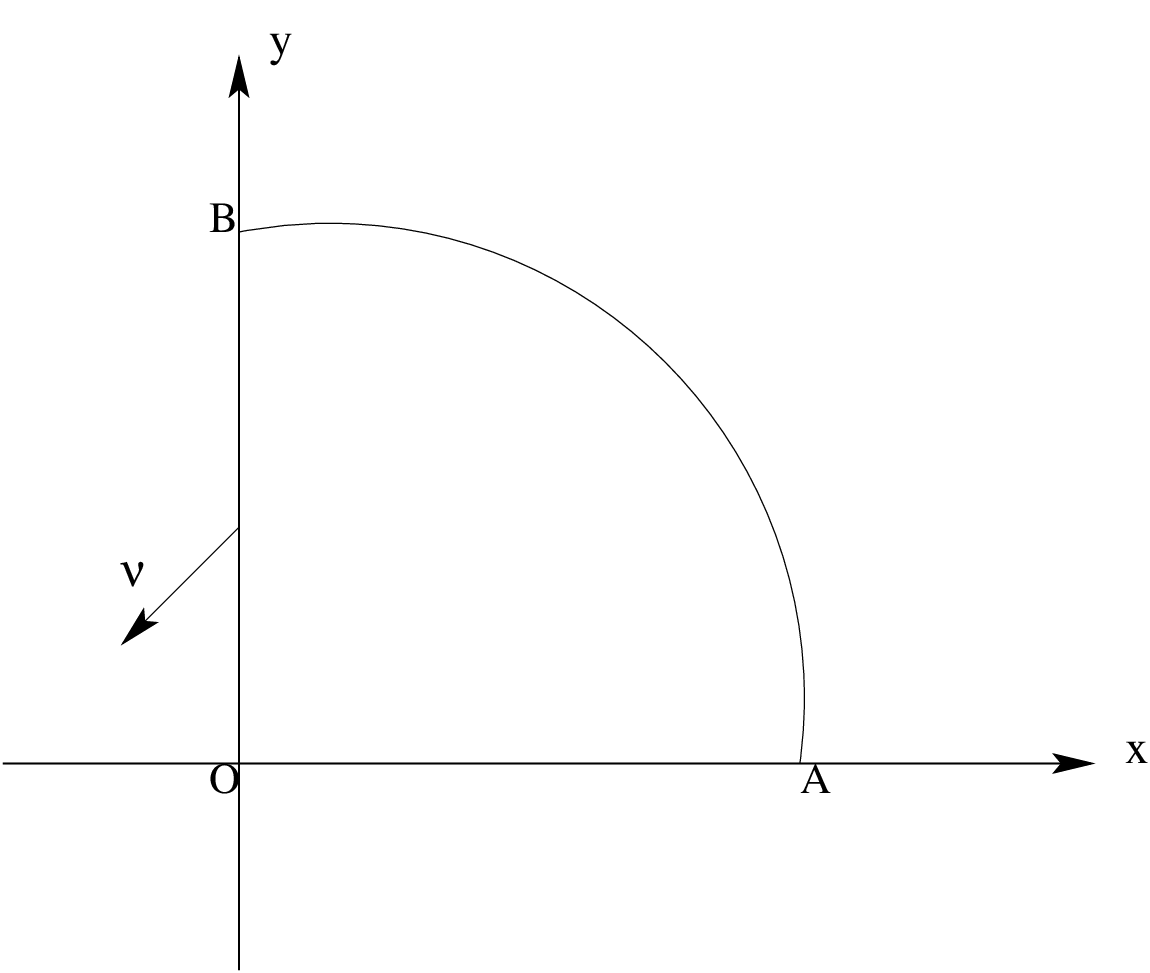}
\end{figure}

Therefore, it requires a further understanding of global
features of the problem,
especially the global relation between the regularity near
the origin and the decay of solutions at infinity,
to ensure the existence of a smooth solution,
more regular than the H\"{o}lder continuity.
A different approach may be required to handle this case.

\appendix
\section{Two Comparison Principles}

In this appendix, we establish two comparison principles.

Suppose that $\Omega$ is a bounded, connected, and open set in
$\R^n$.
Define a uniformly elliptic operator
$$
L\equiv \sum_{i=1,2}\partial_{x_i}\big(\sum_{j=1,2}a_{ij}(\xx)\partial_{x_j} + b_i(\xx)\big)
\qquad
\mbox{in $\Omega$}
$$
in the following sense:
\[
\sum_{i,j=1,2}a_{ij}(\xx)\xi_i\xi_j \ge \gl |\xi|^2\qquad \mbox{for any $\xx \in
\Omega$ and $\xi \in \R^n$},
\]
where $\gl$ is a positive constant. Assume that $a_{ij}, b_i \in
C^1(\Omega)\cap C(\bar{\Omega})$.

\begin{theorem} \label{thm-max1}
 Suppose that $v, w \in C^2(\Omega) \cap
C(\bar{\Omega})$ satisfy
\begin{enumerate}
\item [\rm (i)]
$L v \ge 0$ and $Lw \le 0$ in $\Omega$;
\item[\rm (ii)]
$w>0$ in $\bar{\Omega}$.
\end{enumerate}
Then $\frac{v}{w}$ achieves its positive maximum on the boundary:
\begin{equation}\label{ine-compare1}
\sup_{\Omega}\Big(\frac{v}{w}\big)\le \sup_{\partial \Omega}
\Big(\frac{v^+}{w}\Big).
\end{equation}
\end{theorem}

\begin{proof}
Let \[V=\frac{v}{w}, \qquad B_i = 2\sum_{j=1,2} a_{ij}\frac{w_{x_j}}{w} +b_i. \]
By calculation, we have
\begin{equation}\label{eqal-compare}
    \sum_{i,j=1,2}(a_{ij}V_{x_i})_{x_j} + \sum_{i=1,2}B_i V_{x_i} + \frac{Lw}{w}V =
    \frac{Lv}{w}.
\end{equation}
By assumption, we know that $ \frac{Lw}{w} \le 0$ and
$\frac{Lv}{w}\ge 0$. Therefore, by the weak maximum principle,
Theorem 8.1 in \cite{gilbarg}, we conclude \eqref{ine-compare1}.
\end{proof}

\begin{theorem}\label{thm-max2}
 Suppose that $v, w \in C^2(\Omega) \cap
C(\bar{\Omega})$ satisfy
\begin{enumerate}
\item[\rm (i)] $L v \ge Lw$  and $Lw < 0$ in $\Omega$;
\item[\rm (ii)] $w>0$ in $\bar{\Omega}$.
\end{enumerate}
Then $\frac{v}{w}$ achieves its positive maximum on the boundary or
no greater than $1$ in $\Omega$:
\begin{equation}\label{ine-compare2}
\sup_{\Omega}\Big(\frac{v}{w}\Big)\le \max\Big\{\sup_{\partial
\Omega}\Big(\frac{v^+}{w}\Big), 1\Big\}.
\end{equation}
\end{theorem}

\begin{proof}
Equation \eqref{eqal-compare} implies
  \begin{eqnarray}
  \sum_{i,j=1,2}(a_{ij}V_{x_i})_{x_j} + \sum_{i=1,2}B_i V_{x_i}&=& \frac{Lw}{w}(1-V) +
    \frac{L v - L w}{w}\nonumber\\
    &\ge& \frac{Lw}{w}(1-V). \label{eqal-compare2}
\end{eqnarray}
Assume that $V$ achieves the maximum value $M>1$ at some interior
point $\xx_0 \in\Omega$. Then, by continuity of $V$, there exists a
ball $B_M \equiv B_r(\xx_0) \subset \Omega$ such that
\begin{eqnarray*}
&&\sup_{B_M} V= \sup_{\Omega} V =M >1, \\[1mm]
&&V>1 \qquad\mbox{in $B_M$}.
\end{eqnarray*}
Therefore,
$$
\frac{Lw}{w}(1-V) >0 \qquad\mbox{in $B_M$},
$$
and \eqref{eqal-compare2} implies
\begin{equation}\label{ine-positive-L}
\sum_{i,j=1,2}(a_{ij}V_{x_i})_{x_j} + \sum_{i=1,2}B_i V_{x_i} >0 \qquad \mbox{in} \,\,\, B_M.
\end{equation}
By the strong maximum principle, Theorem 8.19 in \cite{gilbarg}, we
conclude
$$
V \equiv M \qquad\mbox{in $B_M$}.
$$
This implies
that
\[
\sum_{i,j=1,2}(a_{ij}V_{x_i})_{x_j} + \sum_{i=1,2}B_i V_{x_i} =0\qquad \mbox{in} \,\,\, B_M,
\]
which contradicts \eqref{ine-positive-L}. This completes the proof.
\end{proof}

\section{The Shock Polar}
We consider the uniform constant transonic flows with horizontal
incoming supersonic flows. We now employ
the Rankine-Hugoniot conditions \eqref{con-RH1}--\eqref{con-RH4}
to derive a criterion for different arcs $\wideparen{TS}$ and
$\wideparen{TH}$ on the shock polar.

Assume that $U^-$ and $U$ are constant supersonic and subsonic
states, respectively. The shock-front is a straight line: $y_1= s y_2$.
Let $k=\frac{u_2}{u_1}$ and $k^-= \frac{{u_2}^-}{{{u_1}^-}} =0$.
Then the Rankine-Hugoniot conditions \eqref{con-RH1}--\eqref{con-RH4} give
rise to
\begin{eqnarray}\label{con-RH-polar1}
 &&\big[\frac{1}{\rho u_1}\big]=-k s,\\[1mm]
 && \big[ u_1 + \frac{p}{\rho u_1}\big]=-pk s,
  \label{con-RH-polar2}\\[1mm]
&&u_1k= [\,p\,] s,
\label{con-RH-polar3}\\[1mm]
&& \big[ \frac{1}{2}u_1^2(1+k^2) + \frac{\gc p}{(\gc-1)\rho}\big]=
0. \label{con-RH-polar4}
\end{eqnarray}
From \eqref{con-RH-polar3},  $s=\frac{u_1k}{[\,p\,] }$.
Replacing $s$ in \eqref{con-RH-polar1} and \eqref{con-RH-polar2}, we
obtain
\begin{eqnarray}\label{con-RH-polar01}
 &&\big[\frac{1}{\rho u_1}\big][\,p\,]+u_1k^2=0,\\[1mm]
 && \big[ u_1 + \frac{p}{\rho u_1}\big][\,p\,]+ u_1pk^2=0.
  \label{con-RH-polar02}
\end{eqnarray}
From equations \eqref{con-RH-polar4}--\eqref{con-RH-polar02},
we can solve $\rho, u_1$, and $k$ in terms of
$p$.
Regarding $(\rho, u_1, k)$ as functions of $p$, we
differentiate \eqref{con-RH-polar4}--\eqref{con-RH-polar02}
with respect to $p$ to obtain
\begin{equation}\label{eqn-Up}
 B X=f,
\end{equation}
where
\begin{eqnarray*}
&&B= \left(%
\begin{array}{ccc}
  -\frac{[p]}{\rho^2 u_1}, & k^2- \frac{[p]}{\rho u_1^2},  & 2 u_1 k \\[2mm]
  -p\frac{[p]}{\rho^2 u_1}, &pk^2- \frac{p[p]}{\rho u_1^2}+[p],  &2p u_1 k \\[2mm]
  \frac{\gc p}{(\gc-1)\rho^2}, & -u_1(k^2+1), & -u_1^2k \\
\end{array}%
\right),  \\[2mm]
&&X=(\rho_p, (u_1)_p, k_p)^\top,\\[1mm]
&&f=(%
  -[ \frac{1}{\rho u_1} ],
  -[ u_1 + \frac{p}{\rho u_1}]
  -\frac{[\,p\,]}{\rho u_1}+u_1k^2,
 \frac{\gc }{(\gc-1)\rho})^\top.
\end{eqnarray*}
We solve equation \eqref{eqn-Up} for $k_p$  to obtain
\begin{equation}\label{exp-kp}
  k_p = -\frac{\rho\, C_p}{C_0},
\end{equation}
where
\begin{eqnarray}
&&C_p=[p]\big(c^2 + (\gc -1)q^2 - \gc u_1^2\big)
 + (\gc -1)\rho q^2 u_2^2
 + \Big[\frac{1}{\rho u_1} \Big] \rho^2 c^2 u_1 q^2,\\[1mm]
&& C_0 =u_1^3u_2\rho^2\big((\gc+1)p + (\gc-1)p^-\big).
\end{eqnarray}
Notice that $C_0 >0$. Then, when state $U$ belongs to
$\wideparen{TS}$, we find that $k_p >0$, which is equivalent to
\begin{equation}\label{B.11}
C_p <0.
\end{equation}
On $\wideparen{TH}$, we obtain that $C_p >0$.

\section*{Acknowledgments} Gui-Qiang Chen's research was supported in
part by
by NSF Grant DMS-0807551,
the UK EPSRC Science and Innovation award to the Oxford Centre for Nonlinear PDE (EP/E035027/1),
the UK EPSRC Award to the EPSRC Centre for Doctoral Training
in PDEs (EP/L015811/1), and the Royal Society--Wolfson Research Merit Award (UK).
The work of Mikhail Feldman was
supported in part by the National Science Foundation under Grants
DMS-1101260 and DMS-1401490.

\bigskip
\medskip

\end{document}